\documentclass[12pt]{extarticle}
\usepackage{amsmath, amsthm, amssymb, hyperref, color}
\usepackage{graphicx}
\usepackage{caption}
\usepackage{subcaption}
\usepackage{mathtools}
\usepackage{enumerate}
\usepackage{stackengine}
\usepackage{ dsfont }
\usepackage[all]{xypic}
\usepackage{verbatim}
\usepackage{upgreek}
\usepackage{chemfig, chemnum}
\usepackage{tikz}
\usepackage[lined,commentsnumbered,ruled]{algorithm2e}
\usepackage{caption}
\usepackage{subcaption}
\usepackage{float}
\oddsidemargin \evensidemargin
\theoremstyle{theorem}
\newtheorem{theorem}{Theorem}[section]
\newtheorem{proposition}[theorem]{Proposition}
\newtheorem{lemma}[theorem]{Lemma}
\newtheorem{corollary}[theorem]{Corollary}
\newtheorem{conjecture}[theorem]{Conjecture}
\theoremstyle{definition}
\newtheorem{definition}[theorem]{Definition}

\newtheorem{remark}[theorem]{Remark}

\newtheorem{example}[theorem]{Example}

\newcommand{\PP}{\mathbb{P}}
\newcommand{\RR}{\mathbb{R}}
\newcommand{\QQ}{\mathbb{Q}}
\newcommand{\CC}{\mathbb{C} }
\newcommand{\ZZ}{\mathbb{Z}}

\title{\textbf{The Dubrovin threefold of an algebraic curve}}

\author{{Daniele Agostini, T\"urk\"u \"Ozl\"um \c{C}elik,  and Bernd Sturmfels}}

\date{}

\begin{document}

\maketitle

\begin{abstract}
\noindent 
The solutions to the Kadomtsev-Petviashvili equation that arise
from a fixed complex algebraic curve are parametrized
by a threefold in a weighted projective space, which
we name after Boris Dubrovin.
Current methods from nonlinear algebra are applied to study
parametrizations and defining ideals of  Dubrovin threefolds. 
  We highlight the dichotomy between
 transcendental representations and exact algebraic computations.
 Our main result on the algebraic side is a toric  degeneration of
 the Dubrovin threefold into the product of the underlying
 canonical curve and a weighted projective plane.
\end{abstract}

\section{Introduction}\label{sec:Intro}

Let $C$ be a complex algebraic curve of genus $g$.
The associated {\em Dubrovin threefold} $\mathcal{D}_C$ lives
in a weighted projective space $\mathbb{WP}^{3g-1}$,
where the weights are $1$, $2$ and $3$, and each occurs precisely $g$ times.
The homogeneous coordinate ring of
$\mathbb{WP}^{3g-1}$ is the graded polynomial~ring
\begin{equation}
\label{eq:polyring}
 \CC[U,V,W] \,=\,\CC[\,u_1,u_2,\ldots,u_g,\,v_1,v_2,\ldots,v_g,\,w_1,w_2,\ldots,w_g\, ], 
 \end{equation}
where ${\rm deg}(u_i) = 1$, 
${\rm deg}(v_i) = 2$ and
${\rm deg}(w_i) = 3$, for $i=1,2,\ldots,g$.
Points $(U,V,W)$ in the Dubrovin threefold $\mathcal{D}_C$ correspond to
solutions of a nonlinear partial differential equation that describes the motion of
water waves, namely the {\em Kadomtsev-Petviashvili (KP)~equation} 
\begin{equation}
\label{eq:KP1}
 \frac{\partial}{\partial x}\left( 4u_t - 6uu_x -u_{xxx} \right)\, = \, 3u_{yy}.
 \end{equation}
 This differential equation
 represents a universal integrable system in two spatial dimensions, with coordinates $x$ and $y$.
 The unknown function  $u = u(x,y,t)$
describes the evolution in time $t$ of long waves of small amplitude with slow dependence on the transverse coordinate~$y$.

The Dubrovin threefold is an object that appears tacitly in the prominent article \cite{Dub81}
on the connection between integrable systems and Riemann surfaces. Another source
that mentions this threefold is a manuscript on the Schottky problem by
John Little \cite[\S 2]{LittleArXiv}.
Our aim here is to  develop this subject from the 
current perspective of nonlinear algebra \cite{INLA}.

In the algebro-geometric approach to the KP equation (\ref{eq:KP1}) one seeks solutions of the form
\begin{equation}
\label{eq:u_tau}
 u(x,y,t) \,\,=\,\, 2\frac{\partial^2}{\partial x^2} \log \uptau(x,y,t)+c,
\end{equation}
where $c\in\mathbb{C}$ is a constant.
The  function $\uptau = \uptau(x,y,t)$ is known as the {\em $\uptau$-function}
in the theory of integrable systems.
A sufficient condition for (\ref{eq:KP1}) to hold is the quadratic PDE
\begin{equation}\label{eq:hirota}
\!\! ( \uptau_{xxxx}\uptau -4\uptau_{xxx}\uptau_x + 3\uptau_{xx}^2)
\,\,+\,\,4 (\uptau_{x}\uptau_t \,-\,  \uptau \uptau_{xt}) \,\,+\,\, 6 c ( \uptau_{xx}\uptau \, -\, \uptau_x^2)
\,\,+\,\, 3 (\uptau \uptau_{yy}\,-\, \uptau_y^2)\,\,+\,\,8d\uptau^2 
\,\,= \,\,0.
\end{equation}
This is known as {\em Hirota's bilinear form}.
The scalars $ c,d \in \mathbb{C}$ are uniquely determined by~$\uptau$.

One constructs $\uptau$-functions from an algebraic curve $C$ of genus $g$ 
using the method described in \cite{Dub81, DFS}.
Let $B$ be a Riemann matrix for $C$. This is a symmetric $g \times g$ matrix
with entries in $\CC$ whose real part is negative definite. Its construction from $C$ is shown in~\eqref{eq:riemannmatrix}.

 Following \cite[equation (1.1.1)]{Dub81}, we introduce the associated {\em Riemann theta function}
\begin{equation}
\label{eq:RTFreal}
\theta \,=\, \theta({\bf z}\, |\, B)\,\,\, = \,\,\,
\sum_{{\bf u} \in \mathbb{Z}^g} {\rm exp} \left( \frac{1}{2} {\bf u}^T B {\bf u} \, + {\bf u}^T {\bf z} \right).
\end{equation}
The vector of unknowns ${\bf z} = (z_1,z_2,\ldots,z_g)$ is now replaced 
by a linear combination of the vectors $U,V,W$, seen above in
our polynomial ring (\ref{eq:polyring}).
The result is the special $\uptau$-function 
\begin{equation}
\label{eq:thetatau}
 \uptau(x,y,t) \,\, = \,\,\theta( \,Ux + V y + W t + D\,| \, B\,),
 \end{equation}
where $D \in \CC^g$ is also a parameter.
The \emph{Dubrovin threefold} $\mathcal{D}_C$ comprises triples $(U,V,W)$ for which the following condition holds:
there exist $c,d \in \CC$ such that, for all vectors $D \in \CC^g$,
the function (\ref{eq:thetatau}) satisfies the  equation
(\ref{eq:hirota}). This implies that (\ref{eq:u_tau}) satisfies~(\ref{eq:KP1}).
A celebrated result due to Igor Krichever \cite[Theorem 3.1.3]{Dub81} states that such
solutions exist and can be constructed from every smooth point on the curve $C$.
This ensures that $\mathcal{D}_C$ is a threefold.

Krichever's construction of solutions to the KP equation amounts to a  parametrization
of the Dubrovin threefold $\mathcal{D}_C$ by abelian functions. This
transcendental parametrization  will be reviewed in Section~\ref{sec2}.
In Section~\ref{sec3} we present an alternative parametrization,
valid after a linear  change of coordinates in $\mathbb{WP}^{3g-1}$,
that is entirely algebraic. In particular, we shall see that, for curves
$C$ defined over the rational numbers $\mathbb{Q}$, the
 Dubrovin threefold $\mathcal{D}_C$ is also defined over $\mathbb{Q}$.
 We now present a running example which  illustrates this point.

\begin{example} \label{ex:running}
Let $C$ be the genus $2$ curve defined by $y^2=x^6-1$. Its
Dubrovin threefold $\mathcal{D}_C$ lives in  $\mathbb{WP}^5$.
By the construction in Section~\ref{sec3}, it has the algebraic parametrization
\begin{equation}
\label{eq:running2}
\begin{matrix}
u_1=a U_1\,, \,\,
u_2=a U_2\, , \,\,\,
v_1 = a^2 V_1+2 b U_1 \, , \,\,
v_2 = a^2 V_2+2 b U_2 \, ,
\\
w_1 = a^3 W_1+3 a b V_1+c U_1\, , \,\,\,
w_2 = a^3 W_2+3 a b V_2+c U_2,
\end{matrix}
\end{equation}
where $a,b,c$ and $x$ are free parameters, we specify $y = \sqrt{x^6-1}$, and
we abbreviate
\begin{equation}
\label{eq:running3} \!\!\!\!\!\! \begin{matrix}
U_1 = -\frac{1}{y}\, ,\,\,
U_2 =-\frac{x}{y}\, ,\,\,
V_1 = \frac{3x^5}{y^3}\, ,\,\,
V_2 = \frac{2y^2{+}3}{y^3}\, ,\,\,
W_1 =  -\frac{(12y^2{+}27)x^4}{2y^5}\, ,\,\,
W_2  =  -\frac{(6y^2{+}27)x^5}{2y^5}. \end{matrix}
\end{equation}
Every point $(U,V,W)$ on the threefold $\mathcal{D}_C$
gives rise to a $\uptau$-function (\ref{eq:thetatau}) that satisfies (\ref{eq:hirota}).

Using standard implicitization methods  \cite[\S 4.2]{INLA}, we compute
the homogeneous prime ideal of the Dubrovin threefold $\mathcal{D}_C$.
This ideal is minimally generated by the  five polynomials
\begin{equation}
\label{eq:running1} 
\begin{matrix}
3 u_1^5\,+\, 2 u_2^2 w_1-2 u_1 u_2 w_2 +
3 u_1 v_2^2-3 u_2 v_1 v_2\,, \qquad \qquad \qquad \qquad \qquad \qquad
\qquad \qquad \qquad \,\,\, \\
3 u_2^5\,-\,2 u_1^2 w_2+ 2 u_1 u_2 w_1 -
3 u_2 v_1^2+3 u_1 v_1 v_2 \,, \,\,
9 (u_1^4 u_2^4-u_1^4 v_1^2+u_2^4 v_2^2)-4(u_2 w_1 - u_1  w_2)^2, \\
\!\! 9 (u_1^3 u_2^4 v_1{-}u_1^3 v_1^3{+}u_1^4 u_2^3 v_2{+}u_2^3 v_2^3)
+6 (u_1^4 v_1 w_1{-}u_2^4 v_2 w_2)-4 (u_2 w_1 {-} u_1 w_2)(v_2 w_1 {-} v_1 w_2)\,, \\
9 (u_1^2 u_2^4 v_1^2 - u_1^2 v_1^4 + u_1^3 u_2^3 v_1 v_2 + u_1^4 u_2^2 v_2^2+u_2^2 v_2^4)-4 (u_1^4 w_1^2-u_2^4 w_2^2)\\
\quad -6 (u_1^3 u_2^4 w_1-2 u_1^3 v_1^2 w_1+u_1^4 u_2^3 w_2+2 u_2^3 v_2^2 w_2)-4 (v_2 w_1-v_1 w_2)^2.
\end{matrix} 
\end{equation}
These equations are homogeneous of degrees $5,5,8,9,10$, and their coefficients are integers.
\end{example}

Theorem \ref{thm:2quintics} extends the computation in \eqref{eq:running1} to
arbitrary curves of genus two. We pass to genus three in Theorem \ref{thm:quartic},
by computing the prime ideal of $\mathcal{D}_C$ for plane quartics~$C$.

One can also find polynomials that vanish on $\mathcal{D}_C$
directly from the PDE above.  Namely, if we plug the $\uptau$-function
(\ref{eq:thetatau}) into (\ref{eq:hirota}), and we require the result to be identically zero,
then we obtain polynomials in $U,V,W$ and $c,d$ whose coefficients are expressions
in theta constants.  Such equations were derived in
\cite[\S 4.3]{Dub81}. This approach  will be studied in Section~\ref{sec4},
where we focus on the
numerical evaluation of theta constants.
For this we use the package in \cite{julia}.

In Section \ref{sec5} we move on to curves $C$ of arbitrary genus.
Theorem \ref{thm:initialideal} identifies an explicit
initial ideal for the Dubrovin threefold $\mathcal{D}_C$.
 Geometrically, this is a toric
degeneration of $\mathcal{D}_C$ into the product of
the canonical model of $C$
with a weighted projective plane $\mathbb{WP}^2$.
For genus four and higher, we run into the   Schottky problem:
most Riemann matrices $B$ do not arise from algebraic curves.
A solution using the KP equation was given by Shiota \cite{KrSh, Shiota}. 
His characterization of valid matrices $B$ amounts to the
existence of the Dubrovin threefold.

In Section \ref{sec6} we examine special Dubrovin threefolds,
arising from curves that are singular and reducible.
We focus on tropical degenerations, where the 
theta function turns into a finite exponential sum, and
 node-free degenerations, where the theta function
turns into a polynomial \cite{Eies09}. These were studied for
genus three in the context of theta surfaces~\cite[\S 5]{struwe}.

\section{Parametrization by Abelian Functions} \label{sec2}

A parametric representation of the Dubrovin threefold $\mathcal{D}_C$ is given
in \cite[equation (3.1.24)]{Dub81}. This follows Krichever's result
on the algebro-geometric construction of solutions to the KP equation.
 One obtains expressions for
the coordinates of $U,V,W$ by integrating normalized differentials of the second kind
on the curve $C$. In this section we develop this in detail.

We fix a symplectic basis $\,a_1,b_1,\dots,a_g,b_g\,$ for the
first homology group $H_1(C,\mathbb{Z})$ of a 
compact Riemann surface $C$ of genus $g$.
 For any point $p\in C$, we
choose three normalized differentials of the second type, denoted 
$\Omega^{(1)}$, $\Omega^{(2)}$ and $\Omega^{(3)}$. These are meromorphic
differentials on $C$, with poles only at $p$, and with local expansions at $p$ of the~form
\begin{equation}
\Omega^{(1)} = -\frac{1}{z^2}dz+\dots, \qquad \Omega^{(2)} = -\frac{2}{z^3}dz+\dots, \qquad \Omega^{(3)} = -\frac{3}{z^4}dz+ \cdots.
\end{equation}  
Here  $z$ is a local coordinate and the dots denote the regular terms. 
By  \cite[equation~(3.1.24)]{Dub81}, we compute the coordinates of the
vectors $U,V,W$ by integrating these differential forms:
\begin{equation}\label{eq:UVWintegration}
u_i = \int_{b_i} \Omega^{(1)}, \qquad v_i = \int_{b_i} \Omega^{(2)}, \qquad w_i = \int_{b_i} \Omega^{(3)}
\qquad \hbox{for} \,\, i = 1,2,\ldots,g.
\end{equation}

These integrals can be evaluated using a suitable basis for the holomorphic differentials on $C$.
Namely, suppose that $\omega_1,\omega_2,\dots,\omega_g$ is a basis of holomorphic differentials such that 
\begin{equation}\label{eq:adaptedbasis} 
\int_{a_j} \omega_k\,\,=\,\, 2\pi i \cdot \delta_{jk} ,
\end{equation}
where $i = \sqrt{-1}$ and we use the Kronecker $\delta$ notation.
For such a special basis, we can compute a
symmetric Riemann matrix $B = (B_{jk})$ for the curve $C$ by the following integrals:

$$ \quad B_{jk} \,\,=\,  \int_{b_j} \omega_k  \qquad \hbox{for} \,\, 1 \leq j,k \leq g. $$

Around the point $p$ on $C$, we can  write each holomorphic differential in our basis as 
\begin{equation}
\label{eq:holodiff}
\omega_k\,=\,H_k(z)dz,
\end{equation}
where $H_k(z)$ is a holomorphic function of the local coordinate $z$. We
denote the first and second derivative of this function by
$\dot{H}_k(z)$ and $\ddot{H}_k(z)$.
 Our differentials of the second type~are
\begin{equation}
\label{eq:Omega}
 \Omega^{(1)}\,=\,-\omega^{(1)}\,,
\,\,\, \Omega^{(2)}\,=\, -2\omega^{(2)}\,,
\,\,\, \Omega^{(3)} \,=\, -3\omega^{(3)}, 
\end{equation}
 where  $\omega^{(n)}$ has a local expansion $\omega^{(n)} = \frac{1}{z^{n+1}}dz+\dots $. \
  Now, \cite[Lemma 2.1.2]{Dub81} shows that 
\begin{equation}
\label{eq:intloweromega}
\int_{b_j} \omega^{(1)} = H_j(p) , \qquad \int_{b_j} \omega^{(2)} = \frac{1}{2}\dot{H}_j(p) ,
          \qquad \int_{b_j} \omega^{(3)} = \frac{1}{6}\ddot{H}_j(p).
\end{equation}
Putting all of this together, we record the following result:

\begin{proposition} \label{prop:makeKPsol}
Let $p$ be a smooth point on the genus $g$ curve $C$. The vectors in (\ref{eq:UVWintegration})~are
\begin{equation}\label{eq:UVW}
U  \,=\,  -\begin{pmatrix} H_1(p) \\ \vdots \\ H_g(p)  \end{pmatrix}, \quad
V \,=\,  -\begin{pmatrix} \dot{H_1}(p) \\ \vdots \\ \dot{H_g}(p)  \end{pmatrix}, \quad 
W \,=\,  -\frac{1}{2}\begin{pmatrix}\ddot{H_1}(p) \\ \vdots \\ \ddot{H_g}(p)  \end{pmatrix}.
\end{equation}
These formulas specify a $\uptau$-function \eqref{eq:thetatau},
with an arbitrary vector  $D \in \CC^g$,
 such that the~function $u(x,y,t)$ defined in \eqref{eq:u_tau}  is
 a solution to the KP equation \eqref{eq:KP1}
for some constant $c \in \CC$.
\end{proposition}

\begin{proof}
We substitute \eqref{eq:Omega} into \eqref{eq:UVWintegration}, and we 
then use (\ref{eq:intloweromega}) to obtain the formulas  (\ref{eq:UVW}).
The second assertion is Krichever's result \cite[Theorem 3.1.3]{Dub81}
mentioned in the Introduction.
\end{proof}

Proposition \ref{prop:makeKPsol} expresses $(U,V,W)$ as a function
of the point $p $ on the curve $C$.
This defines a map from $C$ into the weighted projective space 
whose coordinate ring equals~\eqref{eq:polyring}:
  \begin{equation}
 \label{eq:happymap}
\! \begin{matrix} C \rightarrow \mathbb{WP}^{3g-1} , \,\,
 p \,\mapsto \,  (U,V,W) \,=\,
 - \bigl(H_1,\dots,H_g,\dot{H}_1,\dots,\dot{H}_g,\frac{1}{2}\ddot{H}_1,\dots,\frac{1}{2}\ddot{H}_g \bigr). 
 \end{matrix} 
 \end{equation}
The image curve is a lifting of the canonical 
model from $\PP^{g-1}$ which we call the {\emph{lifted canonical curve}}.
Indeed, if we consider only the $U$-coordinates in (\ref{eq:UVW}) then these
define the canonical map from $C$ into $\PP^{g-1}$. This is an embedding
if $C$ is not hyperelliptic. For hyperelliptic curves of genus $g \geq 2$,
the canonical map is a $2$-to-$1$ cover of the rational normal curve.

Our construction of lifting the canonical curve from $\PP^{g-1}$ to
$\mathbb{WP}^{3g-1}$ is not intrinsic. The map specified in
\eqref{eq:UVWintegration} or \eqref{eq:UVW} depends
on the choice of a local coordinate $z$ around $p$. 
Suppose we apply an analytic change of local coordinates, $\tilde z=\phi(z)$, around $p$ on $C$.
Then our basis of holomorphic differentials can be expressed, using 
some holomorphic functions $G_i$, as
\begin{equation}\label{eq:changecoordinate}  
\qquad \omega_i \,\,=\,\, G_i(\tilde z)d \tilde z \,\,=\,\, G_i(\phi(z))\phi'(z)dz \,\,\,=\,\, H_i(z)dz
\qquad \hbox{for $\,i=1,2,\ldots,g$.} 
\end{equation}

Using the chain rule from calculus, we find
\begin{equation}
\begin{pmatrix} H_i(p) \\ \dot{H}_i(p) \\ \ddot{H}_i(p) \end{pmatrix} \,\,=\,\,
 \begin{pmatrix} \phi'(0) & 0 & 0 \\ \phi''(0) & \phi'(0)^2 & 0 \\ 
 \phi'''(0) & 3\phi''(0)\phi'(0) & \phi'(0)^3  \end{pmatrix}
 \begin{pmatrix} G_i(p) \\ \dot{G}_i(p) \\ \ddot{G}_i(p) \end{pmatrix}.
\end{equation}
Setting $\,a=\phi'(0),\,b={\phi''(0)}/{2}\,$ and $\,c={\phi'''(0)}/{2}$, we consider the matrix group
\begin{equation}\label{eq:groupG}
G \,\,=\,\, \left\{\, \begin{pmatrix} a & 0 & 0 \\ 2b & a^2 & 0 \\ c & 3ab & a^3  \end{pmatrix} 
\,\,\bigg|\,\,\, a\in \mathbb{C}^*,\,b,c\in \mathbb{C} \,\right\}
\,\, \subset \,\, {\rm GL}(3,\mathbb{C}).
\end{equation}
The group $G$ acts on the weighted projective space $\mathbb{WP}^{3g-1}$ as follows:
\begin{equation}\label{eq:actionG}
\begin{pmatrix} a & 0 & 0 \\ 2b & a^2 & 0 \\ c & 3ab & a^3  \end{pmatrix}\cdot 
\begin{pmatrix} u_i \\ v_i \\ w_i \end{pmatrix} \,\,=\,\, \begin{pmatrix} a u_i \\ 2bu_i+a^2v_i \\ cu_i+3abv_i+a^3w_i
 \end{pmatrix}
 \qquad \hbox{for $\,i=1,2,\ldots,g$.}
\end{equation}

\noindent The orbit of a given point $(\widetilde U,\widetilde V, \widetilde W)$ on $\mathcal{D}_C$ under the action is the surface defined by
 \begin{equation}
\label{eq:orbiteqns}
 \begin{matrix}
     {\tilde u_i} u_j- {\tilde u_j}u_i \, , \,\,
   {\tilde u_i}  ( {\tilde u_j} v_i - {\tilde u_i} v_j)
   - u_i ({\tilde v_i} u_j -{\tilde v_j} u_i ) \, , \,\,\, {\rm and} & \\
     2 {\tilde u_j}^2 ({\tilde u_j} w_i-{\tilde u_i}  w_j)
  -2 u_j^2 ({\tilde w_i} u_j-{\tilde w_j}u_i )
  +3 {\tilde u_j} u_j ({\tilde v_j} v_i - {\tilde v_i} v_j)
  % \end{matrix} 
&  \,\,\,
  \hbox{for} \,\,1 \leq i < j \leq g. \end{matrix}
\end{equation}
 In these equations, the quantities $u_i,v_i,w_i$ are variables while
$\tilde u_i, \tilde v_i, \tilde w_i$ are  complex numbers. This surface 
defined by \eqref{eq:orbiteqns} 
is isomorphic to the weighted projective plane $\mathbb{WP}^2 = \mathbb{P}(1,2,3)$.
The Dubrovin threefold $\mathcal{D}_C$ is
the union of a $1$-parameter family of such surfaces in $\mathbb{WP}^{3g-1}$.
Namely, the union is taken over all points $(\widetilde U, \widetilde V, \widetilde W)$ 
in the image of the map \eqref{eq:happymap}. In Theorem \ref{thm:initialideal}
we shall present a toric degeneration that turns this family into a product.

\begin{remark} \label{rem:theothercomponent}
We here regard $\mathcal{D}_C$ as an irreducible variety. This
is slightly inconsistent with the Introduction,
where $\mathcal{D}_C$ was defined in terms of solutions to the KP equation.
The reason is the involution $(U,V,W) \mapsto (U,-V,W)$ on $\mathbb{WP}^{3g-1}$.
By \cite[(4.2.5)]{Dub81}, this involution preserves the property that 
 \eqref{eq:thetatau} solves \eqref{eq:hirota}, and it swaps
  $\mathcal{D}_C$ with an isomorphic copy $\mathcal{D}_C^{-}$.
Thus the parameter space for solutions
\eqref{eq:u_tau} to \eqref{eq:KP1} would be
the reducible threefold 
 $\mathcal{D}_C \cup \mathcal{D}_C^{-}$.
\end{remark}

In our parametric representation of the Dubrovin threefold  $\mathcal{D}_C$, 
we required holomorphic differentials    (\ref{eq:holodiff})
that are adapted to a symplectic homology basis in the sense of \eqref{eq:adaptedbasis}.  This hypothesis is
essential when we seek solutions to the KP equation. However, it is not needed when
studying algebraic or geometric properties of $\mathcal{D}_C$. For that we allow a
linear change of coordinates. This is important in practice
because differentials $\omega_k$ that satisfy \eqref{eq:adaptedbasis} may
not be readily available. We shall see this in the symbolic computations in Section \ref{sec3}.
In what follows we explain how to work with an arbitrary basis of holomorphic differentials~(\ref{eq:holodiff}).

Let $\widetilde{\omega}_1,\dots,\widetilde{\omega}_g$ be such a basis.
We consider the corresponding $g \times g$ period matrices
\begin{equation}\label{eq:periodmatrices}
\Pi_a \,=\, \left( \int_{a_j} \widetilde{\omega}_i \right)_{ij} \quad
{\rm and} \qquad \Pi_b \,=\, \left( \int_{b_j} \widetilde{\omega}_i \right)_{ij}.
\end{equation} 
From these two non-symmetric matrices, we obtain the symmetric Riemann matrix
\begin{equation}\label{eq:riemannmatrix}
B \,\,=\,\, 2\pi i \cdot  \Pi_a^{-1}\Pi_b .
\end{equation}
For a given plane curve $C$,
the matrices \eqref{eq:periodmatrices}
and \eqref{eq:riemannmatrix} can be computed using numerical methods. 
In our experiments, we used the
{\tt SageMath} implementation due to Bruin et al.~\cite{BruSijZot}.

Given this data, a normalized basis of differentials as in \eqref{eq:adaptedbasis} is obtained as follows:
\begin{equation}
\label{eq:basistransf}
( \omega_1, \omega_2,\ldots,\omega_g)^T \,\,=\,\, \,2\pi i \cdot  \Pi_a^{-1} \,
( \widetilde{\omega}_1, \widetilde{\omega}_2,\ldots, \widetilde{\omega}_g)^T. 
 \end{equation}
The linear transformation in (\ref{eq:basistransf}) induces
a linear change of coordinates on the Dubrovin threefold 
$\mathcal{D}_C$ in $\mathbb{WP}^{3g-1}$
as follows. Consider a smooth
point $p\in C$ and a local coordinate $z$ around~$p$.
We write $\,\widetilde{\omega}_i = \widetilde{H}_i(z)dz$, 
where $\widetilde{H}_i$ is holomorphic,
and we define vectors as in~\eqref{eq:UVW}:
\begin{equation}\label{eq:UVWtilde} \begin{small}
\widetilde{U}  \,:= \, -\begin{pmatrix} \widetilde{H}_1(p) \\ \vdots \\ \widetilde{H}_g(p)  \end{pmatrix}, \quad 
\widetilde{V} \,:=\,  -\begin{pmatrix} \dot{\widetilde{H}_1}(p) \\ \vdots \\ \dot{\widetilde{H}_g}(p)  \end{pmatrix}, \quad 
\widetilde{W} \,:=\,  -\frac{1}{2}\begin{pmatrix}\ddot{\widetilde{H}_1}(p) \\ \vdots \\ \ddot{\widetilde{H}_g}(p)  \end{pmatrix}.
\end{small}
\end{equation}
Then the vectors $U,V,W$ for the adapted basis $\omega_1, \omega_2,\ldots,\omega_g$ are given by
\begin{equation}\label{eq:changecoordinates}
U = 2\pi i \cdot  \Pi_a^{-1} \widetilde{U}, \qquad V = 2\pi i \cdot  \Pi_a^{-1} \widetilde{V}, \qquad W = 2\pi i \cdot  \Pi_a^{-1} \widetilde{W}.
\end{equation}
Any such linear change of coordinates commutes with the action of the group $G$ as in \eqref{eq:actionG}.

\begin{remark}\label{rmk:nicebasis} 
In conclusion, for any basis $\widetilde{\omega}_1,\widetilde{\omega}_2,\dots,\widetilde{\omega}_g$ of holomorphic differentials, we 
can define a Dubrovin threefold in $\mathbb{WP}^{3g-1}$. It is the union of all $G$-orbits of
the points $(\widetilde{U},\widetilde{V},\widetilde{W})$  in \eqref{eq:UVWtilde}.
 Any two such Dubrovin threefolds are related to each other via a linear change of coordinates.
This scenario in $\mathbb{WP}^{3g-1}$ is analogous to that for canonical 
curves in $\mathbb{P}^{g-1}$.
Thus, to compute Dubrovin threefolds and their prime ideals, we can use any
basis $\widetilde{\omega}_1,\widetilde{\omega}_2,\dots,\widetilde{\omega}_g$ of our choosing.
 Whenever this is clear, we  omit the $\,\,\widetilde{}\,\,$  superscript from our notation. However, if
 we want the Dubrovin threefold to parametrize solutions \eqref{eq:u_tau} of the KP equation, 
then we must use a normalized basis as in \eqref{eq:adaptedbasis} or perform the 
linear change of coordinates in~\eqref{eq:changecoordinates}.   
\end{remark}

To illustrate Remark \ref{rmk:nicebasis}, we examine KP solutions 
arising from our running example.

\begin{example}\label{ex:RiemannMatrix}
Consider any zero $\bigl(\widetilde{U},\widetilde{V},\widetilde{W}\bigr)$ 
in $\mathbb{WP}^5$ of the five polynomials in \eqref{eq:running1}.
Here, $\mathcal{D}_C$ was computed using the basis of
differentials $\widetilde{\omega}_1, \widetilde{\omega}_2$ 
given in \eqref{eq:genus2basis} below.   This basis relates to an adjusted basis via
 period matrices which we computed numerically using {\tt SageMath}:
$$
\begin{small}
 \Pi_a  =
\begin{pmatrix}  \phantom{-}1.2143253i &\!  \phantom{-} 1.0516366 + 0.6071627i \\
-1.2143253\, &  \! -0.6071627 - 1.0516366i   
 \end{pmatrix}\!, \,
 \Pi_b = \begin{pmatrix} 
 -1.0516366-0.6071627i & \! 1.2143253i  \\
 -0.6071627-1.0516366i & \! 1.2143253\,
 \end{pmatrix}. 
 \end{small}
 $$
We can now use these matrices to construct two-phase solutions of the KP equation
\eqref{eq:KP1}. First we compute the following Riemann matrix for the curve $\,C = \{ y^2 = x^6-1\}\,$
in Example~\ref{ex:running}:
\[ B \,=\, 2\pi i \cdot \Pi_a^{-1}\Pi_b \,=\, \begin{pmatrix} -7.25519746 & \phantom{-} 3.62759872\, \\
\phantom{-} 3.62759872 & -7.25519746 \, \end{pmatrix} \,\, = \,\,
\frac{2 \pi }{\sqrt{3}} \begin{pmatrix} 2 \! & \! \phantom{-}1 \, \\ 1 \! & \! \phantom{-}2  \,
\end{pmatrix}. \]

This allows us to evaluate the theta function $\theta( {\bf z}|B) $, e.g.~using the
{\tt Julia} package described in \cite{julia}. 
The evaluation is done, for any fixed $D \in \CC^2$, at the points
$Ux + Vy+Wz+D$, where the coefficients $U = (u_1,u_2)^T$,
$V = (v_1,v_2)^T$ and $W = (w_1,w_2)^T$ are obtained
from $\widetilde{U},\widetilde{V},\widetilde{W}$
by the linear change of coordinates in \eqref{eq:changecoordinates}.
The resulting function $u(x,y,t)$ in \eqref{eq:u_tau} solves \eqref{eq:KP1}
for an appropriate constant $c$. In this manner, 
the Dubrovin threefold $\mathcal{D}_C$ that is given by (\ref{eq:running1}) 
represents a $3$-parameter family of two-phase solutions to the KP equation \eqref{eq:KP1}.

For a numerical example, fix $p=(2,\sqrt{63})$  on $C$. The corresponding parameters are 
$$ \begin{small}
U = \begin{pmatrix} 
\phantom{-}0.133702+0.111777i \\ 
-0.059901-0.119802i
\end{pmatrix} \! ,
\,
V = \begin{pmatrix}
-0.151861-0.140376i\\
\phantom{-}0.091278+0.122654i
\end{pmatrix} \! ,\,
W = \begin{pmatrix}
\phantom{-}0.131964 +0.13077i\\ 
-0.094538-0.09780i
\end{pmatrix}\! .
\end{small}
$$
We use the procedure in Remark \ref{rmk:linearcd} below to estimate the constants $c$ and $d$ as follows:
\begin{equation}
\label{eq:cdcd}
c\,=\, 0.02546003+0.15991389i \quad {\rm and} \quad d\,=\,0.00437723+0.00078777i. 
\end{equation}
The corresponding solution $u(x,y,t)$ is complex-valued and it has singularities.
\end{example}

We close this section with an example where the KP solution is  real-valued and regular.

\begin{figure}[H] 
\begin{center} 
\includegraphics[width=0.7\textwidth]{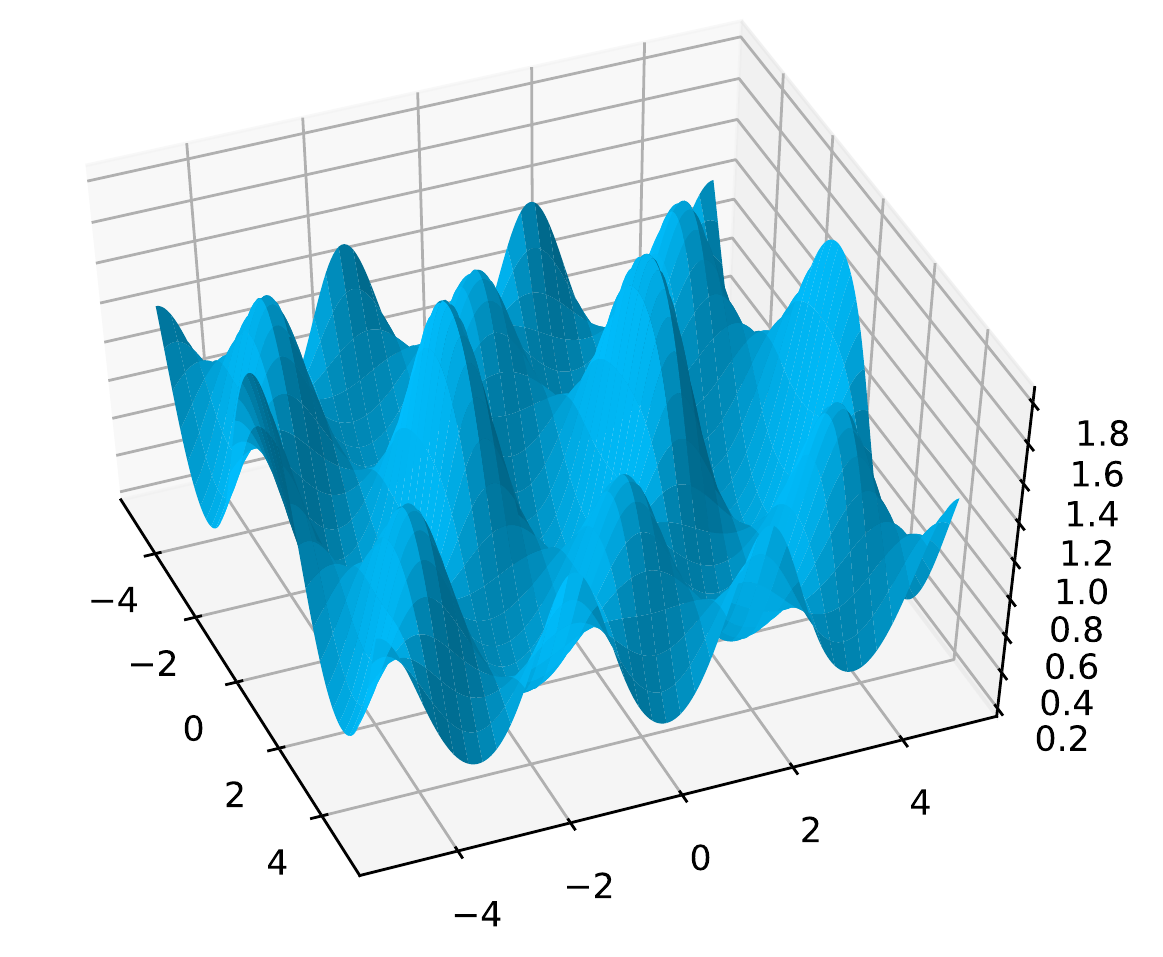}  \vspace{-0.22cm}
\caption{A wave at time $t=0$ derived from the Trott curve.}\label{fig:trott}
\end{center}
\end{figure}
\begin{example}\label{ex:Trott1}
A prominent instance of a plane quartic $C$ is the {\em Trott curve}
\begin{equation}
\label{eq:prominent}
 f(x,y)\,\, =\,\, 144(x^4+y^4)-225(x^2+y^2)+350x^2y^2+81. 
 \end{equation}
This curve is smooth. Its real picture consists of four ovals \cite[Figure 7]{struwe}.
We fix a symplectic basis of paths where the $a_i$ are purely imaginary and the $b_j$ are real. In particular, the paths $b_j$ correspond to ovals of $C$. Using this basis, together with the basis of holomorphic differentials in Example \ref{ex:Trott}, we compute the period matrices and the Riemann matrix:
$$
\begin{small}
\begin{matrix}
\Pi_a &=& \begin{pmatrix}  0.01384015942 & 0.02768031884 & 0.01384015942\\                          
0.01384015941 &           0 &          -0.01384015941 \\
0.02348847438 &           0 &          0.02348847438
 \end{pmatrix} i  \, , \smallskip \\  
B &\,=& \, 
 -2\pi\begin{pmatrix} 
 1.57412534343470 & -0.671587878369476  & -0.230949586695748\\
 -0.671587878369476  & 1.57412534206005  & -0.671587878369476\\
-0.230949586695747 & -0.671587878369476  & 1.57412534343470 \end{pmatrix}.
\end{matrix}
\end{small}
$$
We fix the point $(0,1)$ on $C$, and we compute an associated point on the Dubrovin threefold:
\begin{equation}
\label{eq:pointondubrovin} \begin{small}
\biggl(0\,, \, -\frac{1}{126}\,,\,-\frac{1}{126}\,,\,-\frac{1}{126} \,,\,0\,,\,0\,,\,0\,,\,-\frac{1550}{55566}\,,\,-\frac{1325}{37044} \biggr)
\end{small}
\,\, \in \,\, \mathcal{D}_C \,\, \subset \,\, \mathbb{WP}^8.
\end{equation}
The corresponding  KP solution $u(x,y,t)$  is real-valued and has no singularities.
The graph of the function $\,\RR^2 \rightarrow \RR,\, (x,y) \mapsto u(x,y,0)$, up to a translation, is shown in Figure \ref{fig:trott}.

Equations that define $\mathcal{D}_C$ are given towards the end of Example~\ref{ex:Trott}.
Every point $(U,V,W)$ with $U \not= 0$ that satisfies \eqref{eq:sixfortrott} gives rise to a KP solution like
Figure~\ref{fig:trott}. Note that the homogenenized Trott
quartic $f \bigl(u_1/u_3,u_2/u_3 \bigr) u_3^4$ is a linear combination
of the three polynomials in \eqref{eq:sixfortrott}. For an analytic perspective on the defining ideal of
the threefold $\mathcal{D}_C$ see Example~\ref{ex:Trott2}.
\end{example}

\section{Algebraic Implicitization for Plane Curves} \label{sec3}

This section is concerned with algebraic representations of Dubrovin threefolds.
 The parametrization described in Section \ref{sec2} will be made explicit for
  plane curves, in a form that is suitable for symbolic computations.
  This sets the stage for implicitization  \cite[\S 4.2]{INLA}.
  In Theorems \ref{thm:2quintics} and \ref{thm:quartic} we determine
the prime ideals of all implicit equations for
genus two and three respectively.   The connection to the KP equation requires a basis change 
for the holomorphic differentials, as explained in Remark \ref{rmk:nicebasis}.  
Our varieties give rise to KP solutions after that basis change.
Bearing this in mind, we can now safely omit the $\,\,\widetilde{}\,\,$ superscripts.

Let $C$ be a compact Riemann surface, represented by 
a possibly singular curve $C^o$ in the complex affine plane $\CC^2$. 
The curve $C^o$ is defined by an irreducible polynomial $f(x,y)$.
We assume that $f$ lies in $\QQ[x,y]$, i.e.~all coefficients of $f$ are rational numbers.
This ensures that the computations in this section can be carried out
symbolically over $\QQ$. Our first goal is to start from $f$ and find
a parametric representation of the Dubrovin threefold  $\mathcal{D}_C$ 
in $\mathbb{WP}^{3g-1}$.

Let  $d = {\rm degree}(f)$ and $g = {\rm genus}(C) \leq \binom{d-1}{2}$.
 We assume throughout that
$y^d$ appears with nonzero coefficient in $f$. 
The partial derivatives
of $ f$ are denoted by $f_x=\partial{f}/\partial{x}$ and $f_y=\partial{f}/\partial{y}$. 
Similarly, we write $f_{xx}, f_{xy}, f_{yy}, f_{xxx}, \ldots\,$ for higher-order derivatives.
 We often view the
$y$-coordinate on $C^o$ as a (multi-valued) function of $x$.
It is then denoted $y(x)$.  

As in~\cite[\S 3.1]{BruSijZot}, we choose polynomials
 $h_1,h_2,\ldots,h_g \in \QQ[x,y]$ such that   the following set of holomorphic differentials is a 
   basis for the complex vector space $ H^0(C, \Omega^1_C)$:
\begin{equation}
\label{eq:algbasis1}
\,\, \Big\{ \,\omega_i = H_i(x,y) dx \,\, \big| \,\, i=1,2,\ldots , g \,\Big \} \quad\,\,
{\rm where} \,\,\,\, H_i = \frac{h_i(x,y)}{f_y(x,y)}.
\end{equation}
 
\begin{example}\label{ex:planeCurve}
Let $f$ be a general polynomial of degree $d$ in $\QQ[x,y]$. The curve
$C^o$ is smooth, we have $g = \binom{d-1}{2}$,
and our Riemann surface $C$ is the closure of $C^o$ in $\PP^2$.
For the numerator polynomials $h_1,\ldots,h_g$ we take
the monomials of degree at most $d-3$.
Thus, \eqref{eq:algbasis1} becomes
\begin{equation}
\label{eq:algbasis2}
 \left\{\,  \, \frac{x^iy^j}{f_y} dx\,\,\, \big| \,\,\, 0\leq i+j \leq d-\,3 \right\} .
\end{equation}
Working modulo the principal ideal $\langle f \rangle$, we regard  $y = y(x)$ as an algebraic function in $x$.
\end{example}

We now record the following important consequence of Proposition \ref{prop:makeKPsol}.

\begin{corollary}\label{cor:AlgParam} The formulas for the column vectors $U, V, W$ in \eqref{eq:UVW},
composed with the group action in \eqref{eq:actionG}, give an algebraic
parametrization of the Dubrovin threefold $\mathcal{D}_C$.
This parametrization has coordinates in
$K[a,b,c]$, where $K $ is the field of fractions of $\,\QQ[x,y]/\langle f \rangle $.
\end{corollary}

\begin{proof}
The functions $H_i$ in the differentials $\omega_i$ are rational  in $x$ and $y$,
so they define elements in the function field $K$ of the curve $C$. We write $y = y(x)$,
so $x$ is our local parameter on $C$. We use dot notation for derivatives with
respect to $x$. By implicit differentiation, we find
\begin{equation}\label{eq:dotyquartic} \dot{y}\,=\,-\frac{f_x}{f_y}
\qquad {\rm and} \qquad \ddot{y} \,=\,
-\frac{f_x^2f_{yy}-2f_yf_xf_{xy}+f_y^2f_{xx}}{f_y^3}.
 \end{equation}
We similarly take derivatives of $H_i = H_i(x,y(x))$ with respect to $x$. This yields
rational formulas for  the coordinates $ H_i,\dot{H}_i, \ddot{H}_i$
of $U,V,W$. Again, we view these as elements in~$K$.~\end{proof}

\begin{example}[Smooth plane curves]\label{example:smoothplanecurves}
Let $C$ be a smooth plane curve of degree $d$ described by an affine equation $f(x,y)=0$. 
For the basis of holomorphic differentials we take (\ref{eq:algbasis2}), but with negative signs.
The $g$ coordinates of the canonical curve are elements in $K$:
$$ U \,=\, (u_1,u_2,\ldots,u_g) \,\, = \,\,
-\frac{1}{f_y} \bigl( 1,x,y,x^2,xy,\ldots, y^{d-3}\bigr). $$
The first coordinates of $V$ and $W$ are derived from those of $U$ by implicit differentiation:
	\begin{align*}
	&v_1\,=\, \dot{u}_1 \,=\,  \frac{f_yf_{xy}-f_xf_{yy}}{f_y^3},  \smallskip \\
	&w_1\,=\,\frac{\ddot{u}_1}{2} \,=\,
	\frac{-3 f_x^2 f_{yy}^2
 + f_x^2 f_y f_{yyy}
+6 f_x f_y f_{xy} f_{yy}
-2 f_x f_y^2 f_{xyy}
- 2 f_y^2 f_{xy}^2
 - f_y^2 f_{xx} f_{yy}
 + f_y^3 f_{xxy}}{2f_y^5}.
	\end{align*} 
	The other coordinates of $V$ and $W$ can be computed by applying the
	product rule and implicit differentiation to the formula
	$\,u_k = x^i y^j \cdot u_1$. We thus derive the formulas
	\begin{align*}
	&v_k=\dot{u}_k = \,\dot{u}_1x^iy^j+
	u_1 x^{i-1} y^{j-1} (iy+jx \dot{y})
	 = \frac{(f_yf_{xy}-f_xf_{yy})x^iy^j-if_y^2 x^{i-1}y^j + jf_xf_y x^i y^{j-1}}{f_y^3}. \\
	&w_k \,=\,\frac{\ddot{u}_k}{2}\, \,=\,\,\frac{\ddot{u}_1}{2}x^iy^j  \,+\,  \dot{u_1}(ix^{i-1}y^j+jx^iy^{j-1}\dot{y})
	\\ & \qquad \qquad \qquad +\frac{u_1}{2} 
	\bigl(\,i(i-1)x^{i-2}y^j+2ijx^{i-1}y^{j-1}\dot{y}+j(j-1)x^{i}y^{j-2}\dot{y}^2+jx^{i}y^{j-1}\ddot{y}\,\bigr).
	\end{align*}
All of these expressions are rational functions in $x$ and $y$ with coefficients in $\QQ$.
We regard them as elements in $K$, the field of fractions of $\QQ[x,y]/\langle f \rangle $.
 In practice, this means replacing the numerator polynomial
in $\QQ[x,y]$ and reducing it to a normal form modulo the ideal $\langle f \rangle$.	
 \end{example}

\begin{remark} \label{rmk:polypara}
In the parametrization described in Corollary \ref{cor:AlgParam}, we can 
choose all coordinates to be polynomials. Indeed, consider the formulas derived in 
Example~\ref{example:smoothplanecurves}. We see that $U,V,W$ have  the common denominators $f_y, f_y^3, f_y^5$ respectively. We can thus clear all denominators 
in $(U,V,W)$ by multiplying each coordinate by an appropriate power of $f_y$,
in a manner that does not change the corresponding point in $\mathbb{WP}^{3g-1}$.
Thereafter we apply the action of the group $G$.
We obtain a polynomial map with the
same image $\mathcal{D}_C$ in $\mathbb{WP}^{3g-1}$.
\end{remark}

In our computational experiments we started out with
 canonical curves of genus three and with
hyperelliptic curves.  The following two examples illustrate these classes of curves.

\begin{example}\label{ex:Trott}
Let $C$ be the Trott curve which was studied in Example~\ref{ex:Trott1}. 
The following three differential forms 
constitute our basis for the space $ H^0(C, \Omega^1_C)$ we chose in
 Example~\ref{ex:planeCurve}:
\begin{equation}\label{eq:differentialBasisTrott} 
\omega_1=\frac{1}{f_y} dx, \quad \omega_2=\frac{x}{f_y} dx, \quad \omega_3=\frac{y}{f_y} dx.
 \end{equation}
Note that  $u_1 = -1/f_y$, $u_2 = -x/f_y$ and $u_3 = -y/f_y$, where
$\,f_y = 700x^2y+576y^3-450y$.
 By reducing the formulas for the $V$- and $W$-coordinates in Example~\ref{example:smoothplanecurves}
modulo $f$, we obtain:
$$ \begin{small} \begin{matrix}
&v_1\,=\, 
12 (39556x^3y^2-4650x^3-13950y^2+2025x)f_y^{-3}
, \,\,\,\,\,\quad v_3= 496(638x^2-225)x f_y^{-3} ,  \qquad  \qquad \smallskip \\
& \quad \quad \quad  v_2=
4(79112x^4y^2-13950x^4-13950x^2y^2+6075x^2-3969y^2) f_y^{-3},
\end{matrix} \end{small}
$$ 
$$
\begin{small}
\begin{matrix}
w_1 \, =&
\frac{1}{3}  (450627015168x^{10}+1095273995200x^8y^2-982215036000x^8-1260877167000x^6y^2 \\
&\quad\quad+ 710159081508x^6+430938071100x^4y^2-196724295000x^4-30435203250x^2y^2 \\
&\quad\quad+11445723549x^2-5650169175y^2+2242385775) f_y^{-5} ,\smallskip \\
w_2 \,=& \frac{1}{3}(225313507584x^{11}+547636997600x^9y^2-391782402000x^9-355914999000x^7y^2 \\
&\quad\quad+ 132297850476x^7-102485711700x^5y^2+68233160700x^5+101045778750x^3y^2 \\
&\quad\quad - 45254399085x^3 - 16950507525y^2 + 6727157325) f_y^{-5}, \smallskip \\
w_3\,=& \! \! \frac{62}{3}(25236728x^6{-}14833500x^4{+}1652778x^2{+}297675)
(144x^4{+}350x^2y^2{-}225x^2{-}225y^2{+}81) y f_y^{-5}.
\end{matrix}
\end{small}
$$
To set up our implicitization problem, we first
replace the above formulas for $\,U,V,W\,$ by 
$\,f_y^2U,f_y^4V,f_y^6W\,$ in order to get rid of the denominators. Then, applying
the group action~\eqref{eq:actionG}, we write the coordinates on $\mathcal{D}_C$ as
$\,af_y^2U, \,2bf_y^2U+a^2 f_y^4V, \,cf_y^2U + 3abf_y^4V+ a^3 f_y^6W$.
These are nine polynomials in
five unknowns $a,b,c,x,y$. Regarded modulo the equation $f(x,y) = 0$,
this is the parametrization of the Dubrovin threefold $\mathcal{D}_C$
that is promised in
Remark \ref{rmk:polypara}.
The nine expressions are quite complicated. But they satisfy the
following six nice relations:
\begin{equation}
\label{eq:sixfortrott}
 \begin{small}
\begin{matrix}
450 u_1^2 u_3+450 u_2^2 u_3-324 u_3^3+u_2 v_1-u_1 v_2\,,\,\,
     700 u_1^2 u_2+576 u_2^3-450 u_2 u_3^2+u_3 v_1-u_1 v_3,\ \\
     576 u_1^3+700 u_1 u_2^2-450 u_1 u_3^2-u_3 v_2+u_2 v_3,  \\
     450 u_1 u_3 v_1+450 u_2 u_3 v_2+225 u_1^2 v_3+225 u_2^2 v_3-486 u_3^2 v_3+u_2 w_1-u_1 w_2, \\
     700 u_1 u_2 v_1+350 u_1^2 v_2+864 u_2^2 v_2-225 u_3^2 v_2-450 u_2 u_3 v_3+u_3 w_1-u_1 w_3, \\
     864 u_1^2 v_1+350 u_2^2 v_1-225 u_3^2 v_1+700 u_1 u_2 v_2-450 u_1 u_3 v_3-u_3 w_2+u_2 w_3.
\end{matrix}     
\end{small}
 \end{equation}
 These are the six relations in
\eqref{eq:minorsderivatives} and
\eqref{eq:minorsderivatives2}.
After saturation, as explained in 
Theorem~\ref{thm:quartic}, these generate the prime ideal of $\mathcal{D}_C$.
All computations are done using Gr\"obner bases over~$\mathbb{Q}$.
\end{example}

We illustrate the situation for hyperelliptic curves in the most basic case of genus two.

\begin{example}\label{ex:cuge2} Let $C$ be a general curve of genus two. It is defined
by an equation
$$y^2 = F(x)\, , \quad {\rm where} \quad F(x)=(x-a_1)(x-a_2)\cdots(x-a_6) 
\quad {\rm with} 	\,\,	a_i\in \mathbb{C}.	$$ 
A basis for the space of differential forms on the hyperelliptic curve $C$ consists of
\begin{equation}
\label{eq:genus2basis}
 \omega_1 \,=\, \frac{1}{\sqrt{F}}dx  \quad {\rm and}
\quad \omega_2 \,=\, \frac{x}{\sqrt{F}}dx .
\end{equation}
We take the first and second $x$-derivative of 	their coefficients.
According to the formulas of \eqref{eq:UVW}, the resulting formula
for the point $(U,V,W) = (u_1,u_2,v_1,v_2,w_1,w_2)$ in $\mathbb{WP}^5$ equals
\begin{equation}
\label{eq:sixtuple}
\left(-\frac{1}{\sqrt{F}},\, -\frac{x}{\sqrt{F}}, \,\frac{F'}{2F\sqrt{F}},\,
\frac{xF'-2F}{2F\sqrt{F}},\,\frac{2FF''-3(F')^2}{8F^2\sqrt{F}},\,
\frac{2xFF''+4FF'-3x(F')^2}{8F^2\sqrt{F}}\right).
\end{equation}
This provides the algebraic parametrization in Corollary~\ref{cor:AlgParam}. In particular, 
if we go back to Example~\ref{ex:running},  then the parametrization in \eqref{eq:sixtuple} becomes the one presented in equation~\eqref{eq:running3}.
\end{example}

We conclude this section with two general theorems 
on the prime ideals defining~$\mathcal{D}_C$. The first one, Theorem \ref{thm:2quintics},
 explains the relations we saw in \eqref{eq:running1}.
Here we work in the polynomial ring $\CC[u_1,u_2,v_1,v_2,w_1,w_2]$ with the grading
 ${\rm deg}(u_1) = {\rm deg}(u_2) = 1$,
 ${\rm deg}(v_1) = {\rm deg}(v_2) = 2$, and
 ${\rm deg}(w_1) = {\rm deg}(w_2) = 3$. The group $G$ acts on
 this ring via \eqref{eq:actionG}. The following  two 
 polynomials of degree five
 are invariant (up to scaling by $a^5$) under this $G$-action:
	$$ \begin{matrix} 
	\qquad I_1 \,=\,   2u_2^2 w_1 - 2u_1 u_2 w_2 + 3 u_1 v_2^2  - 3 u_2 v_1 v_2,\\
	\qquad I_2 \,=\,   2u_1^2 w_2 - 2u_1 u_2 w_1 + 3 u_2 v_1^2  - 3 u_1 v_1 v_2 . 
\end{matrix}
$$

The tuple \eqref{eq:sixtuple} parametrizes an
 algebraic curve in $\mathbb{C}^6$.
The orbit of this curve under the group $G$ is a $4$-dimensional
variety in $\mathbb{C}^6$. Its image
in $\mathbb{WP}^5$ is the Dubrovin threefold $\mathcal{D}_C$.
We write $\,\bar F(u_1,u_2) \,=\, F(u_2/u_1) u_1^6\,$
for the binary sextic obtained by homogenizing $F(x)$.

\begin{theorem} \label{thm:2quintics}
Let $C$ be a genus two curve, represented
by a sextic $F(x)$ as in Example \ref{ex:cuge2}.
	There are two linearly independent quintics that vanish on the
	Dubrovin threefold $\mathcal{D}_C$:
	\begin{equation}
	\label{eq:twoquintics}
	\partial \bar F / \partial u_1 \,-\, 2 I_1  \qquad
	{\rm and} \qquad
	\partial \bar F / \partial u_2 \,-\, 2 I_2 . 
	\end{equation}
	The prime ideal of the Dubrovin threefold is minimally generated
	by five polynomials, of degrees $5,5,8,9,10$. This ideal
	is obtained from that generated by (\ref{eq:twoquintics}) via
	saturating  $\langle u_1,u_2 \rangle$.
\end{theorem}

\begin{proof}
This is verified by a computer algebra over the field $L = \QQ(a_1,a_2,\ldots,a_6)$.
The coordinates of the parametrization \eqref{eq:sixtuple} live in $L[x,y]/\langle y^2 - F(x) \rangle $,
and we check that \eqref{eq:twoquintics} vanishes when we substitute
\eqref{eq:sixtuple} for $(U,V,W)$. The two equations in (\ref{eq:twoquintics})
define a variety of codimension two in $\mathbb{WP}^5$. The irreducible threefold
$\mathcal{D}_C$ must be one of its irreducible components. We saturate \eqref{eq:twoquintics} 
by $\langle u_1,u_2 \rangle$, thus removing an extraneous nonreduced component on $\{U= 0\}$.
The result of this saturation is an ideal with five minimal generators of degrees 
$5,5,8,9,10$. A further computation in {\tt Macaulay2} verifies that this ideal is prime. 
This implies that the irreducible threefold defined by this prime ideal must be equal to $\mathcal{D}_C$.
\end{proof}

We next turn to genus three. Assuming that $C$ is non-hyperelliptic, its canonical
model is a quartic in $\PP^2$ with coordinates $U=(u_1,u_2,u_3)$.
We seek the prime ideal of the  Dubrovin threefold $\mathcal{D}_C$ in $\mathbb{WP}^8$.
This ideal lives in the polynomial ring $\CC[U,V,W]$ in nine variables.
Here $\CC$ can be replaced by the subfield over which $C$ is~defined.
For us, this is usually $\QQ$.
The next theorem explains the relations in~\eqref{eq:sixfortrott}.
It should be compared with Lemma \ref{lem:trailingterms}.

\begin{theorem}\label{thm:quartic}
Let $C$ be a smooth algebraic curve of genus three given by a ternary quartic $f(u_1,u_2,u_3)$. The prime ideal of its Dubrovin threefold is minimally generated by 17 polynomials  in $\CC[U,V,W]$. These have degrees 
$3,3,3,4,4,4,5,5,5,8,8,9,9,10,10,11,12$. The first six generators suffice 
up to saturation by $\langle u_1,u_2,u_3 \rangle$. The three cubic ideal generators are 
\begin{equation}\label{eq:minorsderivatives}
\frac{\partial f}{\partial u_1} \, + \, \begin{vmatrix} u_2 & v_2 \\ u_3 & v_3 \end{vmatrix}, 
\qquad \frac{\partial f}{\partial u_2}\, - \,\begin{vmatrix} u_1 & v_1 \\ u_3 & v_3 \end{vmatrix}, 
\qquad \frac{\partial f}{\partial u_3}\, + \,\begin{vmatrix} u_1 & v_1 \\ u_2 & v_2 \end{vmatrix}.
\end{equation}
Fixing the quintic $g = v_1 \cdot {\partial{f}}/{\partial u_1} + 
v_2 \cdot {\partial{f}}/{\partial u_2} + v_3 \cdot {\partial{f}}/{\partial u_3}$,
 the three quartic generators~are 
\begin{equation}\label{eq:minorsderivatives2}
\frac{\partial g}{\partial u_1}  \,+\, 2\begin{vmatrix} u_2 & w_2 \\ u_3 & w_3 \end{vmatrix}, 
\qquad \frac{\partial g}{\partial u_2} \,- \,2 \begin{vmatrix} u_1 & w_1 \\ u_3 & w_3 \end{vmatrix}, 
\qquad \frac{\partial g}{\partial u_3} \,+ \,2\begin{vmatrix} u_1 & w_1 \\ u_2 & w_2 \end{vmatrix}.
\end{equation}
\end{theorem}

The cubics (\ref{eq:minorsderivatives}) imply that
the quartic $f(u_1,u_2,u_3)$  is in the ideal. This
is consistent with the general theory (cf.~\cite[\S 4.3]{Dub81}) since 
the quartic $\{f=0\}$ in~$\PP^2$ is a canonical curve.

\begin{proof}
Equations \eqref{eq:minorsderivatives}
and \eqref{eq:minorsderivatives2} can be proved
by a direct computation for a general quartic with
indeterminate coefficients. That symbolic computation is facilitated by
the fact that these equations are invariant under the
action by ${\rm PGL}(3,\CC)$ on the vectors $U,V,W$ and the
induced action on quartics in $u_1,u_2,u_3$. So, it suffices
to check a six-dimensional family of quartic curves
that represents a Zariski dense set of  orbits.
Alternatively,   \eqref{eq:minorsderivatives}
can be derived geometrically
by examining the meaning of the vectors $U,V$ in $\CC^3$.
The vector $U$ represents a point on the canonical curve which
we identify with $C$ itself. To understand the meaning of $V$,
we examine the parametrization of $\mathcal{D}_C$ given in \eqref{eq:UVW}.
The plane spanned by $U$ and $V$ in $\mathbb{C}^3$ corresponds to the tangent line
 to $C$ at the point $U$. By Cramer's rule, this implies that the gradient of $f$ at $U$
 equals up to a scalar multiple the vector whose three
 coordinates are the determinants on the right hand side of 
 \eqref{eq:minorsderivatives}. A computation reveals that 
 this scalar multiple  equals $1$. This explains \eqref{eq:minorsderivatives}, and then \eqref{eq:minorsderivatives2}
 follows from the proof of Lemma \ref{lem:trailingterms}.
 
The variety defined by our equations in $\mathbb{WP}^8$ is
irreducible of dimension three
outside the locus defined by $\langle u_1,u_2,u_3 \rangle$.
This follows immediately from the structure of the equations. 
First, we have an irreducible curve in the $\PP^2$ with coordinates
$U$. For every point on that curve,
 \eqref{eq:minorsderivatives} gives two independent linear equations
 in $V$. After choosing $U$ and $V$, \eqref{eq:minorsderivatives} gives two independent linear equations
 in $W$. Thus there are three degrees of freedom, and our variety is irreducible
 since the equations over $C$ are linear.
 The precise number and degrees of minimal generators for the
 prime ideal are obtained by computations with generic~$f$.
 \end{proof}
 
 The next section offers an alternative numerical view on implicitizing Dubrovin threefolds.
A conceptual explanation of Theorem~\ref{thm:quartic}, valid for arbitrary genus $g$,
appears in Theorem~\ref{thm:initialideal}.

\section{Transcendental Implicitization} \label{sec4}

Our workhorse in this section is the Riemann theta function $\theta({\bf z}\,|\,B)$, defined by its
series expansion in \eqref{eq:RTFreal}. We chose the Riemann matrix $B$
 as in \cite{Dub81}, with negative definite real~part.
This differs by a factor $2 \pi i$ from the Riemann matrix
 in the algebraic geometry literature. The latter is also used in the
 {\tt Julia} package~\cite{julia}.
This section builds on \cite[\S IV.2]{Dub81}. We show how theta series lead to
polynomials in $\CC[U,V,W]$ that vanish on the Dubrovin threefold of
a curve $C$ of genus $g$ with Riemann matrix $B$. A point $(U,V,W)$
lies in this threefold  if and only if the $\uptau$-function of \eqref{eq:thetatau} satisfies Hirota's bilinear relation \eqref{eq:hirota} for some $c,d\in \mathbb{C}$ and any $D\in \mathbb{C}^g$.
In this definition, by the Dubrovin threefold we mean the union $\mathcal{D}_C \cup \mathcal{D}_C^{-}$
referred to in Remark \ref{rem:theothercomponent}. Hence, the irreducible
threefold $\mathcal{D}_C$, studied algebraically in Section~\ref{sec3},  occurs in 
two different sign copies in the variety defined here.  For what follows we prefer:

\begin{definition}[Big Dubrovin threefold]
	Let $\mathbb{WP}^{3g+1}$ be the weighted projective space 
	with coordinates $(U,V,W,c,d)$ where the new variables $c,d$ have degrees $2$ and $4$ respectively.  The \emph{big Dubrovin threefold} $\mathcal{D}^{\rm big}_C$ is the set of points in $\mathbb{WP}^{3g+1}$ such that 
	Hirota's bilinear relation \eqref{eq:hirota}
	is satisfied for the function $\uptau(x,y,t)=\theta(Ux+Vy+Wt+D)$, where $D\in \mathbb{C}^g$ is arbitrary. 
\end{definition}

We begin by showing that the big Dubrovin threefold is indeed an algebraic variety.
Given any ${\bf z}$ in $\CC^g$, we write $\theta({\bf z}) = \theta({\bf z} | B)$ 
for the complex number on the right in (\ref{eq:RTFreal}).
We set $\partial_U := u_1 \frac{\partial}{\partial z_1} + \cdots + u_g \frac{\partial}{\partial z_g}$,
and we define $\partial_U \theta ({\bf z})$ to be the value at ${\bf z}$ of that directional
derivative of the theta function. This value is a linear form in $U $ with complex coefficients.
We similarly define 
$\partial_V \theta ({\bf z})$ and $\partial_W \theta ({\bf z})$. We also
consider the values of higher order derivatives like
$\,\partial_U^4\theta({\bf z}),  \partial^2_V\theta({\bf z}), \partial_U\partial_W \theta({\bf z})$.
These are homogeneous polynomials of degree four in $\CC[U,V,W]$.

For any fixed vector $\mathbf{z}\in\mathbb{C}^g$, we define the \emph{Hirota quartic} $H_{\bf z}$
 to be the expression
\begin{multline}\label{eq:hirotaquartics}
 \left(\partial_U^4\theta({\bf z}) \cdot \theta({\bf z})-4\partial^3_U\theta({\bf z})\cdot \partial_U\theta({\bf z})+3\{\partial^2_U\theta({\bf z})\}^2\right)+4\cdot \left(\partial_U\theta({\bf z})\cdot \partial_W\theta({\bf z})-\theta({\bf z}) \cdot\partial_U\partial_W \theta({\bf z})\right)\\+6c\cdot \left( \partial^2_U\theta({\bf z})\cdot \theta({\bf z}) - \{\partial_U\theta({\bf z})\}^2 \right) + 3\cdot \left(\theta({\bf z})\cdot \partial^2_V\theta({\bf z})-\{\partial_V\theta({\bf z})\}^2 \right) +8d\cdot \theta({\bf z})^2.
\end{multline}	
This is a homogeneous quartic in the polynomial ring $\CC[U,V,W,c,d]$, where $c$  and $d$ have degrees $2$ and~$4$
respectively.
The coefficients of the Hirota quartic $ H_{\bf z}(U,V,W,c,d)$ depend on the values of 
the theta function $\theta$ and its partial derivatives at ${\bf z}$. The zero set of $H_{\bf z}$
is in the weighted projective space $\mathbb{WP}^{3g+1}$.
We consider the intersection of these hypersurfaces.

\begin{proposition}\label{prop:bigDubrovin}
The big Dubrovin threefold $\mathcal{D}^{\rm big}_C$ is the intersection 
in $\mathbb{WP}^{3g+1}$  of the  algebraic hypersurfaces defined by the
Hirota quartics $H_{\bf z}$, as $\mathbf{z}$
runs over all vectors in $ \mathbb{C}^g$.  
\end{proposition}

\begin{proof}
If we expand the left hand side of Hirota's bilinear relation (\ref{eq:hirota}), then
we obtain the expression $H_{\bf z}(U,V,W,c,d)$ for ${\bf z} = Ux+Vy+Wz+D$. 
By definition, a point $(U,V,W,c,d) \in \mathbb{WP}^{3g+1}$ belongs to $\mathcal{D}^{\rm big}_C$ 
if and only if this expression is zero for all $x,y,t \in \CC$ and all $D\in \mathbb{C}^g$. 
Since $D$ is arbitrary in $\mathbb{C}^g$, the value of $\mathbf{z}=Ux+Vy+Wt+D$ is also arbitrary.
This implies that $\mathcal{D}^{\rm big}_{C}$ is the intersection of all the 
hypersurfaces $\{H_\mathbf{z} = 0 \}$ as $\mathbf{z}$ ranges over $\mathbb{C}^g$.
\end{proof}

For any ${\bf z} \in \CC^g$, we can compute the Hirota quartic $H_{\bf z}$
using numerical software for evaluating theta functions and their derivatives.
The state of the art for such software is the {\tt Julia} package introduced in \cite{julia}.
This was used for the computations reported in this section.
These  differ greatly from the exact symbolic computations
reported in Section~\ref{sec3}.

One drawback of Proposition \ref{prop:bigDubrovin} is
that the number of Hirota quartics is infinite.
We next derive  a finite set of equations for $\mathcal{D}^{\rm big}_{C}$ via
the addition formula for theta functions with characteristics. 
This derivation was explained by Dubrovin in \cite[\S IV.1]{Dub81}. 

A \emph{half-characteristic} is an element $\varepsilon \in (\mathbb{Z}/2\mathbb{Z})^g$ 
which we see as a vector with $0,1$ entries.  
 Given any two half-characteristics $\varepsilon,\delta \in (\mathbb{Z}/2\mathbb{Z})^g$, 
 their \emph{theta function with characteristic}  is
  \begin{equation}
\label{eq:thetaFunctionChar}
\theta\begin{bmatrix} \varepsilon \\ \delta \end{bmatrix}({\bf z}\, |\, B)\,\,\, = \,\,\,
\sum_{{\bf u} \in \mathbb{Z}^g} {\rm exp} \left( \frac{1}{2} ({\bf u}+\varepsilon)^T B ({\bf u}+\varepsilon) + ({\bf z} + 2\pi i \delta)^T B ({\bf u} + \varepsilon) \right).
\end{equation}
When $\varepsilon=\delta=0$, this function is precisely the Riemann theta function \eqref{eq:RTFreal}. 
For $\delta = 0$ but arbitrary $\varepsilon $,
we consider \eqref{eq:thetaFunctionChar} with the doubled period matrix~$2B$.
This is abbreviated by
\begin{equation}
\label{eq:doubled}
\hat{\theta}[\varepsilon]({{\bf z}})\,\, :=\,\, {\theta}\begin{bmatrix} \varepsilon \\ 0 \end{bmatrix}({\bf z}\, |\, 2B) .
\end{equation}
We are interested in the values $\,\hat{\theta}[\varepsilon](0)\,$
of these $2^g$ functions at ${\bf z}=0$. For fixed $B$, these are
complex numbers, known as \emph{theta constants}.
We use the term theta constant also for evaluations at ${\bf z}=0$ of derivatives of (\ref{eq:doubled}) 
with the vector fields $\partial_U$, $\partial_V$, $\partial_W$ as above.
With these conventions, the following expression is a polynomial of degree four
in $(U,V,W,c,d)$:
\begin{equation}\label{eq:dubrovinquartic}
F[\varepsilon] \,\,:=\,\, \partial^4_U \hat{\theta}[\varepsilon](0)-\partial_U\partial_{W}\hat{\theta}[\varepsilon](0)+\frac{3}{2}c\cdot \partial^2_U\hat{\theta}[\varepsilon](0)+\frac{3}{4}\partial^2_V\hat{\theta}[\varepsilon](0)+d\hat{\theta}[\varepsilon](0).
\end{equation}
We call $F[\varepsilon](U,V,W,c,d)$  the {\em Dubrovin quartic} associated with the half-characteristic $\varepsilon$. 

There are $2^g$ Dubrovin quartics in total, one for each half-characteristic 
$\varepsilon$  in $ (\mathbb{Z}/2\mathbb{Z})^g$.
We find that these quartics can also be used as implicit equations for the
big Dubrovin threefold $\,\mathcal{D}_C^{\rm big}\,$ inside $\,\mathbb{WP}^{3g+1}$.
This follows by combining  Proposition \ref{prop:bigDubrovin} with the following result.

\begin{proposition}[Dubrovin] \label{prop:bigdubrovin}
The Dubrovin quartics and the Hirota quartics span  the same vector subspace of $\,\CC[U,V,W,c,d\,]_4$.
This space of quartics defines the big Dubrovin threefold.
\end{proposition}

\begin{proof}
This result is essentially the one proved in \cite[Lemma 4.1.1]{Dub81}.
A key step is the~identity
	\begin{equation}\label{eq:quarticlinearcombination}
	 H_{\bf z} \,\,\, = \,\,\,\, 8\cdot \!\! \!\sum_{\varepsilon \in 
	 (\mathbb{Z}/2\mathbb{Z})^g} \! \hat{\theta}[\varepsilon](2{\bf z}) \cdot F[\varepsilon] \qquad
	 {\rm in} \,\,\,\,\CC[U,V,W,c,d\,]_4.
	\end{equation}
This is proved via the addition formula \cite[(4.1.5)]{Dub81}, 
 analogously to the derivation from (4.1.8) to (4.1.10) in \cite{Dub81}.
We can also invert the linear relations in (\ref{eq:quarticlinearcombination}) and express
the Dubrovin quartics $F[\varepsilon]$ 
as $\CC$-linear combinations
of the Hirota quartics $H_{\bf z}$. Indeed, the functions $z\mapsto \hat{\theta}[\varepsilon](2z)$ are 
linearly independent \cite[(1.1.11)]{Dub81}.
If we take ${\bf z}$ from any fixed set of
$2^g$ general points  in $\CC^g$ then the corresponding 
$2^g \times 2^g$ matrix of theta constants $\hat{\theta}[\varepsilon](2{\bf z})$ are invertible.
The second assertion in Proposition \ref{prop:bigdubrovin} now follows from
Proposition \ref{prop:bigDubrovin}.
\end{proof}

\begin{corollary}
The Dubrovin threefold $\mathcal{D}_C$ is an irreducible
component of the image of the big Dubrovin threefold $\mathcal{D}^{\rm big}_C$ under
the map  $\, \mathbb{WP}^{3g+1} \rightarrow   \mathbb{WP}^{3g-1},\,
(U,V,W,c,d) \mapsto (U,V,W) $.
\end{corollary}

\begin{proof}
The image of the big Dubrovin threefold $\mathcal{D}^{\rm big}_C$ 
in $ \mathbb{WP}^{3g-1}$ is
also an algebraic threefold. Its equations are obtained from those of
$\mathcal{D}^{\rm big}_C$ by eliminating the unknowns $c$ and $d$.
In fact, this image is equal to $\mathcal{D}_C \cup \mathcal{D}_C^{-}$.
The second component is explained in Remark \ref{rem:theothercomponent}.
This follows from the characterization of the Dubrovin threefold
as a  space of KP solutions. 
\end{proof}

\begin{remark}\label{rmk:linearcd}
Consider any point $(U,V,W)$ in the Dubrovin threefold $\mathcal{D}_{C} \subset
\mathbb{WP}^{3g-1}$. It is given to us numerically. The point
 has a unique preimage in the 
big Dubrovin threefold $\mathcal{D}^{\rm big}_C \subset \mathbb{WP}^{3g+1}$.
To compute that preimage, we plug $(U,V,W)$ into several Hirota quartics~$H_{\bf z}$
or Dubrovin quartics $F[\varepsilon]$. This results in an overdetermined
system of linear equations in two unknowns $c$ and $d$. We use numerical
methods to approximately solve these equations. This gives us estimates for $c$ and $d$.
This method was used to estimate the constants in \eqref{eq:cdcd}.
\end{remark}

Suppose now that we are given a curve $C$ by way of its Riemann matrix $B$.
For instance, this is the hypothesis for the construction of KP solutions in \cite{DFS}.
We can compute approximations of the quartics $H_{\bf z}$ or $F[\varepsilon]$
using numerical software for theta functions (cf.~\cite{julia}).
  Following \cite[\S 4.3]{Dub81}, we can then find 
  equations for the canonical model of $C$ in $\mathbb{P}^{g-1}$. 
  For any half-characteristic $\varepsilon \in (\mathbb{Z}/2\mathbb{Z})^g$, 
  we write $Q[\varepsilon]$ for  the Hessian matrix 
  of the function~$\hat{\theta}[\varepsilon]({\bf z})$. We regard $Q[\varepsilon]$
    as a quadratic form in $g$ variables. The next lemma characterizes
    the intersection of  the subspace of quartics in Proposition  \ref{prop:bigdubrovin}
    with the polynomial subring $\CC[U,V,W]$.

\begin{lemma}\label{lemma:quarticelimination}
A $\,\CC$-linear combination $\sum \lambda_{\varepsilon} F[\varepsilon]$ 
of the Dubrovin quartics is independent of the two unknowns $c$ and $d$ if and only if it has the form
	\begin{equation}\label{eq:quarticelimination} 
	\sum_{\varepsilon \in (\mathbb{Z}/2\mathbb{Z})^g} \lambda_{\varepsilon} \cdot \partial^4_U\hat{\theta}[\varepsilon] 
	\end{equation}
	where the $2^g$ complex scalars $\lambda_{\varepsilon}$ satisfy the linear equations
	\begin{equation}\label{eq:linearcondition}
	\sum_{\varepsilon} \lambda_{\varepsilon} \cdot Q[\varepsilon] \,=\, 0 \qquad {\rm and}  \qquad
	\sum_{\varepsilon} \lambda_{\varepsilon} \cdot \hat{\theta}[\varepsilon] \,=\, 0.
	\end{equation} 
The linear system \eqref{eq:linearcondition} has maximal rank, i.e.~it has $\,2^g-\frac{g(g+1)}{2}-1$  independent 
solutions~$(\lambda_{\varepsilon})$.
\end{lemma}

\begin{proof}	 Using the quadratic forms $Q[\varepsilon]$, 
we can rewrite the Dubrovin quartic \eqref{eq:dubrovinquartic} as follows:
\begin{equation}
 F[\varepsilon] \,\,= \,\, \partial^4_U \hat{\theta}[\varepsilon] \,-\, Q[\varepsilon](U,W)
 \, +\,\frac{3}{2}c \, Q[\varepsilon](U,U) \,+\,\frac{3}{4}Q[\varepsilon](V,V) \, + \,  d\hat{\theta}[\varepsilon].
\end{equation}
The  linear combinations of the $F[\varepsilon]$ where the variables $c,d$ do not appear have the form
$$
\begin{small}
\sum_{\varepsilon} \lambda_{\varepsilon} \partial^4_U\hat{\theta}[\varepsilon] - \sum_{\varepsilon}\lambda_{\varepsilon}Q[\varepsilon](U,W)+\frac{3}{2}c \cdot \sum_{\varepsilon} \lambda_{\varepsilon} Q[\varepsilon](U,U)+\frac{3}{4}\sum_{\varepsilon} \lambda_{\varepsilon} Q[\varepsilon](V,V)+
d\sum_{\varepsilon} \lambda_{\varepsilon}\hat{\theta}[\varepsilon],
\end{small}
$$
where $\sum_{\varepsilon} \lambda_{\varepsilon}\hat{\theta}[\varepsilon]=0$ and $(\sum_{\varepsilon} \lambda_{\varepsilon} Q[\varepsilon])(U,U)=0$. The second condition means that
the quadratic form $\sum_{\varepsilon} \lambda_{\varepsilon} Q[\varepsilon]$ is zero, and hence
 $(\sum_{\varepsilon} \lambda_{\varepsilon}
Q[\varepsilon])(U,W)=(\sum_{\varepsilon} \lambda_{\varepsilon}Q[\varepsilon])(V,V)=0$.
 This proves the first part of the lemma. For the second part, we need to show that the linear system 
 \eqref{eq:linearcondition} for the $(\lambda_{\varepsilon})$ is of maximal rank.
 This is proven in \cite[Lemma 4.3.1]{Dub81}.
\end{proof}

As a consequence, we can use the KP equation to reconstruct a curve $C$ from its Riemann matrix
$B$. This suggests a numerical solution to the Schottky Recovery Problem (cf.~\cite[\S 2]{CKS}).

\begin{proposition}\label{prop:recovery}
	The quartics \eqref{eq:quarticelimination} cut 
	out the canonical model of the curve $C$ in $\mathbb{P}^{g-1}_U$.
\end{proposition} 

\begin{proof}
We know from \cite[\S 4.3]{Dub81}
 that the projection of the big Dubrovin threefold $\mathcal{D}^{\rm big}_C$ 
 to $\mathbb{P}^{g-1}_U$ corresponds to the canonical model of $C$. Hence, 
any linear combination of the $F[\varepsilon]$~where only the variables $u_1,\dots,u_g$ appear vanishes
on the curve. For the converse, we need to show that amongst the 
quartics \eqref{eq:quarticelimination} we can find equations that cut out the canonical curve. 
  Following \cite[\S 4.3]{Dub81}, we  will find these amongst  Hirota quartics  $H_{\mathbf{z}}$ 
   in which  only the variables $u_1,\dots,u_g$ appear. Indeed, any such quartic is a linear 
   combination of the $F[\varepsilon]$, thanks to Proposition~\ref{prop:bigdubrovin}.
   Lemma \ref{lemma:quarticelimination} shows that it has the form \eqref{eq:quarticelimination}. 
   Suppose that $\mathbf{z}\in \mathbb{C}^g$ is a singular point of the theta divisor $\{ \theta(\mathbf{z})=0 \}$.
   This means
    $\theta(\mathbf{z})=\frac{\partial \theta}{\partial z_1}(\mathbf{z})=
   \cdots = \frac{\partial \theta}{\partial z_g}(\mathbf{z})=0$. Then
      all but one of the terms in \eqref{eq:hirotaquartics} vanish: 
     $\,H_{\bf z} = 3(\partial^2_U\theta(\mathbf{z}))^2$.
   This depends only on $u_1,\ldots,u_g$, and
any such quartic is of the form \eqref{eq:quarticelimination}.
  We assume that the curve $C$ has genus at least five and is
     not hyperelliptic or trigonal.
    The other cases require special arguments, which we omit here.
  By results of   Petri \cite{GL} and Green \cite{Green}, we know that the canonical ideal of $C$ is generated by the quadrics $\partial^2_U\theta(\mathbf{z})$, where $\mathbf{z}$ varies on the singular locus of the theta divisor. Hence, the quartics $(\partial^2_U\theta(\mathbf{z}))^2$ cut out the canonical model of $C$ as a set.
\end{proof}

\begin{example} \label{ex:Trott2}
Let $g=3$ and consider the period matrix $B$ in Example~\ref{ex:Trott1}.
We apply the numerical process above to recover the Trott curve $C$
up to a projective transformation of~$\PP^2$.
We use the {\tt Julia} package in \cite{julia} to numerically evaluate the theta constants
	 for the given $B$. This allows us to write down the
	eight Dubrovin quartics $F[\varepsilon]$, and from this we obtain the
	system \eqref{eq:linearcondition} of seven linear equations in eight unknowns $\lambda_\varepsilon$.
         Up to scaling, this system has a 
	unique solution	  \eqref{eq:quarticelimination}, as promised 
	by Proposition \ref{prop:recovery}.  Computing this solution is
	equivalent to evaluating the $8 \times 8$ 
determinant given by Dubrovin in \cite[equation  (4.2.11)]{Dub81}.

	The quartic we obtain from this process does not look like the Trott quartic at all:
\begin{equation}
\label{eq:secretTrott}
	\begin{small}
	\begin{matrix}
	-0.04216205642716586u_1^4 + 0.12240048937276882u_1^3u_2  -0.29104871408187094u_1^3u_3 \\
	-6.8912949529273355u_1^2u_2^2 + 17.414377754001833u_1^2u_2u_3 -7.511468695367071u_1^2u_3^2 \\
	-14.027390884600191u_1u_2^3 + 3.264586380028863u_1u_2^2u_3  + 17.414377754001833u_1u_2u_3^2\\
	-0.29104871408187094u_1u_3^3 -7.013695442300095u_2^4 -14.027390884600202u_2^3u_3 \\
	-6.891294952927339u_2^2u_3^2 + 0.12240048937276349u_2u_3^3 -0.04216205642716675u_3^4.
	\end{matrix}
	\end{small}
\end{equation}
However, it turns out that  \eqref{eq:prominent} and \eqref{eq:secretTrott} are
equivalent under the action of ${\rm PGL}(3,\CC)$ on ternary quartics. We verified this
using the \texttt{Magma} package in \cite{LerRitSij}. Namely, we computed the
 Dixmier-Ohno invariants of both curves, and we checked that they agree up to numerical round-off.
 Any extension to $g \geq 4$ involves the Schottky problem, as discussed in Section \ref{sec5}.
\end{example}

A similar method can be applied to hyperelliptic curves. 
For $g \geq 3$ we  recover a rational normal curve $\PP^1$ in $\PP^{g-1}$.
But, we can also find the branch points of the $2$-$1$ cover~$C \rightarrow \PP^1$.

\begin{remark}\label{rmk:quinticsgenustwo}
	For $g=2$, no nonzero quartics arise from Lemma \ref{lemma:quarticelimination}. However, we can
consider quintics $\sum_{\varepsilon} \ell[\varepsilon]\cdot F[\varepsilon]$ 
where the four $\ell[\varepsilon]$ are unknowns linear forms in $U = (u_1,u_2)$.
We seek such quintics where only the variables $U,V,W$ appear.  This happens if and only if
$$ \begin{matrix} \sum_{\varepsilon} \ell[\varepsilon]\cdot Q[\varepsilon](U,U) \,=\, 0 \qquad {\rm and}
	 \qquad \sum_{\varepsilon}\ell[\varepsilon]\cdot \hat{\theta}[\varepsilon] \,=\, 0. \end{matrix} $$
	 This is a system of six linear equations in eight unknown complex numbers, namely the
	 coefficients of the $\ell[\varepsilon]$.
	 It has two independent solutions,  giving us two quintics
$\,\sum_{\varepsilon} \ell[\varepsilon]\cdot F[\varepsilon]$.
Up to taking linear combinations, these are precisely the two quintics \eqref{eq:twoquintics}
in Theorem  \ref{thm:2quintics}. 
\end{remark}

\begin{example}\label{ex:bigDubrovin} Consider the genus two curve in Example~\ref{ex:running}. 
Its Dubrovin quartics are
$$
\begin{tiny}
 \begin{matrix}
4.4044247813u_1^4 + 8.80884956304u_1^3u_2 + 13.21327434456u_1^2u_2^2 
+8.80884956304u_1u_2^3 + 4.4044247813u_2^4
-0.1673475726606u_1^2c \\ -0.167347572669u_1u_2c -0.16734757266u_2^2c + 0.11156504844u_1w_1+ 0.055782524223u_1w_2
 + 0.055782524223u_2w_1 \\ +0.111565048440 u_2w_2
 -0.083673786330 v_1^2 -0.08367378633v_1v_2 -0.0836737863v_2^2
  + 1.0042389593d , \smallskip \\
13.5267575687u_1^4 + 27.053515137u_1^3u_2 + 20.4046553367u_1^2u_2^2 +6.877897768u_1u_2^3 + 32.6563192498u_2^4
-0.51395499447u_1^2c \\ -0.51395499447u_1u_2c 
-4.95646510162u_2^2c +0.34263666298u_1w_1+ 0.171318331491u_1w_2
 + 0.171318331491u_2w_1  \\ +3.30431006774u_2w_2 
-0.256977497237 v_1^2 -0.256977497237 v_1v_2 -2.47823255081v_2^2
  + 0.33474631977 d , \smallskip   \\  
  32.6563192498 u_1^4 + 6.87789776801u_1^3u_2 + 20.4046553367u_1^2u_2^2 +27.053515137u_1u_2^3 + 13.5267575687u_2^4
  -4.9564651016u_1^2c \\ -0.513954994u_1u_2c  
-0.51395499447u_2^2c +3.3043100677u_1w_1+ 0.17131833149u_1w_2
 + 0.171318331491u_2w_1 \\ +0.342636663u_2w_2 
-2.4782325508 v_1^2 -0.2569774972370v_1v_2 -0.2569774972369v_2^2
+0.334746319778d , \smallskip \\      
  32.6563192505u_1^4 + 123.7473792u_1^3u_2 + 195.7088775358u_1^2u_2^2 +123.7473792u_1u_2^3 + 32.6563192505u_2^4
-4.9564651017u_1^2c \\ -9.398975209u_1u_2c 
-4.9564651017u_2^2c +3.30431006782u_1w_1+ 3.132991736u_1w_2
 + 3.132991736u_2w_1  \\ +3.30431006782u_2w_2 
-2.47823255087 v_1^2-4.699487604v_1v_2 -2.47823255086v_2^2
+0.334746319778d .
\end{matrix}
\end{tiny}
$$
By Proposition \ref{prop:bigdubrovin},
these four quartics $F[\varepsilon]$ cut out the big Dubrovin threefold in $\mathbb{WP}^{7}$. 
By eliminating $c$ and $d$ numerically,  as described in Remark \ref{rmk:quinticsgenustwo},
 we obtain two quintics in $u_1,u_2,v_1,v_2,w_1,w_2$.
A distinguished basis for this space of quintics is given by 
\eqref{eq:twoquintics}, where
$$ \bar{F} \, =\,  (r u_1+s u_2)^6 \,+\, (s u_1+r u_2)^6 , \,\,\,
{\rm with}  \,\,r = 0.5596349    - 0.9693161\, i  \,\,\, {\rm and} \,\,     s = 1.11926985. $$ 
This binary sextic has rank two \cite[\S 9.2]{INLA},
so it is equivalent under the action of ${\rm PGL}(2,\CC)$ to the sextic
$u_1^6 - u_2^6$ we started with in Example~\ref{ex:running}.
Thus, up to a projective transformation of $\PP^1$, 
numerical computation based on Proposition \ref{prop:bigdubrovin}
 recovers the constraints shown in~\eqref{eq:running1}.
\end{example}

\section{Genus Four and Beyond} \label{sec5}

The theta function \eqref{eq:RTFreal} and the theta constants $\,\hat{\theta}[\varepsilon](0)\,$
are defined for any complex symmetric $g \times g$ matrix $B$ with negative
definite real part. Such matrices represent principally polarized abelian varieties of dimension $g$.
We view the moduli space of such abelian varieties as a variety that is
parametrized by theta constants. For each point in that moduli space, 
i.e.~for each compatible list of theta constants, we can study the
$2^g$ Dubrovin quartics $F[\varepsilon]$ in \eqref{eq:dubrovinquartic}.

We here lay the foundation for future studies
of these universal equations. For $g \geq 4$, one big goal is to eliminate
the parameters $U, V,W$, in order to obtain constraints 
among the theta constants that define the Schottky locus.
That this works in theory is a celebrated theorem of
Shiota  \cite{KrSh, Shiota}, but it has never been carried out in practice.
% stating that the solutions of
% the KP equation described above characterize Jacobians of curves among
% all abelian varieties. 
For $g=4$, we hope to recover
the classical Schottky-Jung relation for the 
Schottky hypersurface. Here the canonical curves are space sextics in $\PP^3$.
 For $g=5$, the Schottky locus
has codimension three in the moduli space, and canonical curves are
intersection of three quadrics in $\PP^4$.
It will be very interesting to experiment with that case, ideally building on the advances
in \cite{AC, FGSM}.

\begin{example}[Genus four curves are planar]
When computing parametrizations of Dubrovin threefolds $\mathcal{D}_C$
as in Sections \ref{sec3} and \ref{sec4}, it is convenient to work
with a planar model of the given curve $C$. Planar curves are typically
singular in $\PP^2$, but they can be smooth in other toric surfaces, such as
$\PP^1 \times \PP^1$.  For instance, 
 bicubic curves in $\PP^1 \times \PP^1$ are the general
 canonical curves in genus four.
 The polynomial defining their planar representation~is
\begin{equation}
\label{eq:affine33}
 f(x,y) \quad = \quad \sum_{i=0}^3 \sum_{j=0}^3 c_{ij} x^i y^j . 
 \end{equation}
We fix the basis of four holomorphic differentials 
in \eqref{eq:algbasis2}  by taking  $i,j=0,1$. This implies
$$ U \,=\, (u_1,u_2,u_3,u_4)\,\, =\,\, - (1/f_y) \cdot (1,x,y,xy) . $$
The formulas for $V$ and $W$ are obtained by implicit differentiation as in
Example \ref{example:smoothplanecurves}. The resulting polynomial
parametrization (cf.~Remark \ref{rmk:polypara}) is used in
Example \ref{ex:gennuss4} below. The canonical model of $C$ in
$\PP^3$ is defined by $ u_1 u_4 - u_2 u_3$
and a cubic which we identify with (\ref{eq:affine33}).
For instance, starting with $f = 1 - x^3 - y^3 - x^3y^3$, we arrive at the canonical ideal
$I_C = \langle u_1 u_4 - u_2 u_3 \,,\, u_1^3 - u_2^3 - u_3^3 - u_4^3   \rangle $.
This space sextic is studied in \cite[Example 2.5]{CKS}.
\end{example}

We are interested in the Schottky Recovery Problem \cite[\S 2]{CKS}. This 
asks for the equations of the canonical model of the curve $C$, 
provided the Riemann matrix $B$ is known to lie in the Schottky locus.
There is no equational constraint for the Schottky locus in genus three,
and we can start with any $B$.
We saw this in Example \ref{ex:Trott2}, and it is also a key point in \cite{DFS}.
   For higher genus, Schottky recovery is nontrivial.
   See \cite[Example 2.5]{CKS} and the next illustration.
   
\begin{example}   
For a brief case study in genus four, we consider the symmetric matrix
\begin{equation}
\label{eq:BringB}
B_\tau \,\,\, = \,\,\, 2 \pi i \cdot
\tau \cdot \begin{small} \begin{pmatrix}
 4 & 1 & -1\,\, & \,1 \\
 1 & 4 & 1 & -1 \, \\
 -1\,\, & 1 & 4 & \,1 \\
  1 & -1\, \,& 1 & \, 4 
  \end{pmatrix}, \end{small}
\quad \text{where $\tau \in \CC$.}
\end{equation}
The matrix in \eqref{eq:BringB} appears in \cite[equation (1.1)]{BN} and \cite[Theorem 1]{RR}.
For the appropriate constant $\tau$, it represents the Riemann matrix of a prominent genus four curve, \emph{Bring's curve}.

For any given $\tau$, we can compute the $16$ Dubrovin quartics $F[\varepsilon]$ numerically.
Using elimination steps explained in  Lemma \ref{lemma:quarticelimination},
we derive five quartics in $\CC[u_1,u_2,u_3,u_4]$.
According to Proposition \ref{prop:recovery}, these quartics cut out the
canonical curve in $\PP^3$ set-theoretically. Of course, this assumes that
such a curve actually exists.
This happens when the matrix $B_\tau$ lies
in the hypersurface given by the
Schottky-Jung relation, which is given explicitly in \cite[Theorem~2.1]{CKS}.
It imposes a transcendental equation on the parameter $\tau$. One can solve this equation
numerically, either using the method explained in \cite[Example 2.3]{CKS}, or by exploring 
for which $\tau$ our five quartics have a solution in $\PP^3$.
In this manner, we can verify the  solution
$$ \tau_0 \,= \,
-\,0.502210544891808050269557385637 + 0.933704454903021171789990736772 \, i. $$
This constant is defined by $j(\tau_0) = -25/2$. Here $j$ is the modular function of weight zero
that represents the $j$-invariant of an elliptic curve, given by its familiar Fourier series expansion
$$ j(\tau) \,\,=\,\, q^{-1} + 744 + 196884 q + 21493760 q^2 +\, \cdots 
\qquad {\rm where} \quad q = e^{2\pi i \tau}. $$
Riera and Rodr\'iguez \cite[Theorem 2]{RR} determined the value $\tau_0$ for 
Bring's curve. We computed the digits above with {\tt Magma},
using a hypergeometric function formula
for inverting~$\tau \mapsto j(\tau)$.
\end{example}

In order to develop tools for Schottky recovery, it is vital to gain a better understanding of the
 ideal of the Dubrovin threefold. This is our goal in the remainder of this section.
Let $C$ be a smooth non-hyperelliptic curve of genus $g$,
canonically embedded in  $\PP^{g-1}$. Its ideal
 $I_C$ lives in $ \CC[U] = \CC[u_1,\ldots,u_g]$.
This is a subring of the coordinate ring $\CC[U,V,W]$ of  $\mathbb{WP}^{3g-1}$.
The {\em canonical ideal} $I_C$ of a curve $C$ is a classical topic
in algebraic geometry.  For $g=4$,  $I_C$ is the complete intersection of
a quadric and a cubic. By Petri's theorem \cite{GL}, for $g\geq 5$,
the canonical ideal $I_C$ is generated by quadrics, unless $C$ is trigonal or a smooth plane quintic.

Consider a homogeneous polynomial  $f $ of degree $d$ in $I_C$.
The expression $f( x U + y V + t W)$ is a polynomial of degree $d$
in $x,y,t$, and its $\binom{d+2}{2}$ coefficients are
homogeneous polynomials
in $ \CC[U,V,W]$ whose degrees range from $d $ to $3d$.
The {\em polarization of the canonical ideal}, denoted
${\rm Pol}(I_C)$, is the ideal 
in $\CC[U,V,W]$ that is generated by all these coefficients,
where $f$ runs over any generating set of $I_C$.
We also consider the $g \times 3$ matrix $(UVW)$ whose
columns are $U,V$ and $W$, and we write
$\wedge_2(UVW)$ for the ideal generated by its $2 \times 2$-minors.

Our object of interest is the prime ideal $\,\mathcal{I}(\mathcal{D}_C)
\subset \CC[U,V,W]\,$ of the Dubrovin threefold.
In the next theorem we determine an initial ideal of $\mathcal{I}(\mathcal{D}_C)$.
This initial ideal is not a monomial ideal. It is
specified by a partial term order on $\CC[U,V,W]$.
For an introduction to the relevant theory ({\em Gr\"obner bases} and
{\em Khovanskii bases}), we refer to \cite[\S 8]{KM} and the references  therein.
We fix the partial term order given by the following weights on the~variables:
\begin{equation}
\label{eq:weight3}
{\rm weight}(u_i) = 0,\,\, 
{\rm weight}(v_i) = 1, \,\,
{\rm weight}(w_i) = 2 \quad {\rm for} \quad i=1,2,\ldots,g.
\end{equation}
The passage from $\mathcal{I}(\mathcal{D}_C)$ to the  {\em canonical initial ideal}
${\rm in}\bigl(\mathcal{I}(\mathcal{D}_C)\bigr)$ corresponds to
a toric degeneration of the Dubrovin threefold.
Our result states, geometrically speaking, that the variety of 
${\rm in}\bigl(\mathcal{I}(\mathcal{D}_C)\bigr)$
 is the product of the
canonical curve and a weighted projective plane
$\mathbb{WP}^2$. This  threefold might serve
as a combinatorial model for approximating KP~solutions.

\begin{theorem} \label{thm:initialideal}
The canonical initial ideal of $\,C$ is prime.
It is generated by the
polarization of the canonical ideal together with the
constraints that $U$, $V$ and $W$ are parallel. In symbols,
\begin{equation}
\label{eq:inIDC}
{\rm in}\bigl(\mathcal{I}(\mathcal{D}_C)\bigr) \,\,\,= \,\,\,
{\rm Pol}(I_C) \,\,+\,\, \wedge_2(UVW).
\end{equation}
The integral domain
$\CC[U,V,W]/{\rm in}\bigl(\mathcal{I}(\mathcal{D}_C)\bigr) $
is the Segre product of the canonical ring $\CC[U]/I_C$ of the
genus $g$ curve $C$
with a polynomial ring in three variables
that have degrees $1,2,3$.
\end{theorem}

Before we present the proof of this theorem,
we discuss its implications in low genus.

\begin{example}[$g=3$]
By Theorem \ref{thm:quartic}, the ideal $\,\mathcal{I}(\mathcal{D}_C)\,$ has
$17$ minimal generators. Its initial ideal 
${\rm in}\bigl(\mathcal{I}(\mathcal{D}_C)\bigr)$
has $24$ minimal generators,
namely the $15 = \binom{4+2}{2}$ equations obtained by polarizing the
ternary quartic that defines $C$ in $\PP^2$, and the nine $2 \times 2$ minors of $(UVW)$.
The following piece of {\tt Macaulay2} code computes the
two ideals above for the Trott curve:
\begin{small}
\begin{verbatim}
R = QQ[u1,u2,u3,v1,v2,v3,w1,w2,w3,
       Degrees => {1,1,1,2,2,2,3,3,3},  Weights => {0,0,0,1,1,1,2,2,2}];
f = 144*u1^4+350*u1^2*u2^2-225*u1^2*u3^2+144*u2^4-225*u2^2*u3^2+81*u3^4;
g = diff(u1,f)*v1 + diff(u2,f)*v2 + diff(u3,f)*v3;
I = ideal(diff(u1,f)+u2*v3-u3*v2,diff(u2,f)+u3*v1-u1*v3,diff(u3,f)+u1*v2-u2*v1,
diff(u1,g)+2*(u2*w3-u3*w2),diff(u2,g)-2*(u1*w3-u3*w1),diff(u3,g)+2*(u1*w2-u2*w1));
\end{verbatim}
The ideal {\tt I} is generated by \eqref{eq:minorsderivatives}
and \eqref{eq:minorsderivatives2}. We next compute
$\,\mathcal{I}(\mathcal{D}_C)\,$ via the
saturation step in Theorem~\ref{thm:quartic}.
Thereafter we display $\,{\rm in}\bigl(\mathcal{I}(\mathcal{D}_C)\bigr)$.
This verifies Theorem~\ref{thm:initialideal} for the Trott curve:
\begin{verbatim}
IDC = saturate(I,ideal(u1,u2,u3));
codim IDC, degree IDC, betti mingens IDC
inIDC = ideal leadTerm(1,IDC); toString mingens inIDC
codim inIDC, degree inIDC, betti mingens inIDC, isPrime inIDC
\end{verbatim}
\end{small}
Each of the $24$ minimal generators of {\tt inIDC}
arises as the initial form of a polynomial in {\tt IDC}.
\end{example}

\begin{example}[$g=4$] \label{ex:gennuss4}
We represent a genus four canonical curve in $\PP^3$ 
by the quadric $q = u_1 u_4 - u_2 u_3$ and a
general cubic $f = f(u_1,u_2,u_3,u_4)$. Consider their Jacobian matrix
$$ J \quad = \quad 
\begin{pmatrix}
 \frac{\partial q}{\partial u_1} &
 \frac{\partial q}{\partial u_2} &
 \frac{\partial q}{\partial u_3} &
 \frac{\partial q}{\partial u_4} \smallskip \\
 \frac{\partial f}{\partial u_1} &
 \frac{\partial f}{\partial u_2} &
 \frac{\partial f}{\partial u_3} &
 \frac{\partial f}{\partial u_4} 
\end{pmatrix},
$$
and let $J_{kl}$ denote the determinant
of the $2 \times 2$ submatrix of $J$ with column indices $k$ and $l$.
The generator of lowest degree in $\, \mathcal{I}(\mathcal{D}_C)\,$ 
is the quadric $q$. In degree three, there are eight minimal generators:
the cubic $f$, the polarization 
$\,u_1 v_4+u_4 v_1 -u_2 v_3- u_3 v_2\,$ of $q$, 
as well as
\begin{equation}
\label{eq:jacobrel}
\begin{matrix}
u_1 v_2-u_2 v_1-J_{34}\,, \,\,
u_1 v_3-u_3 v_1+J_{24}\,,\,\,
u_2 v_3-u_3 v_2-J_{14}\,, 
\\
u_1 v_4-u_4 v_1-J_{23}\,, \,\,
u_2 v_4-u_4 v_2+J_{13}\,, \,\,
u_3 v_4-u_4 v_3-J_{12}.\,
\end{matrix}
\end{equation}
These six cubics illustrate the principle of behind our degeneration
to the canonical initial ideal.
The $2 \times 2$ minors $u_k v_l - u_l v_k$
are the initial forms with respect to
the weights (\ref{eq:weight3}) 
of the polynomials in (\ref{eq:jacobrel})
because the trailing term $J_{kl}$  
only involves the variables $u_1,u_2,u_3,u_4$.

The initial ideal (\ref{eq:inIDC}) has
$34$ minimal generators, namely
the six polarizations of $q$, the 
ten polarizations of $f$, the
$18 $ minors of the $4 \times 3$ matrix $(UVW)$.
Each of these $34$ polynomials is in fact the
initial form of a minimal generator of
the Dubrovin ideal $\mathcal{I}(\mathcal{D}_C)$.
For instance, for the quartic $u_i w_j - u_j w_i$
we add many trailing terms of the forms
$u_i u_j v_k$ and $u_i u_j u_k u_l$.
The largest degree of a minimal generator
is nine. It arises from the polynomial
$\, f(w_1,w_2,w_3,w_4) $.
\end{example}

We now embark towards the proof of Theorem \ref{thm:initialideal}
with a sequence of three lemmas.
Our standing assumption is that $C$ is smooth and non-hyperelliptic.
The Dubrovin threefold $\mathcal{D}_C$ is constructed as in Section \ref{sec3}.
Given the polynomial $f(x,y)$ defining $C^o$,
we view $y$ as a function of $x$,
we consider the differential forms
$\omega_i$ in \eqref{eq:algbasis1},  and we 
form the derivatives $\dot{H_i}, \,\ddot{H_i}$ 
of $H_i(x) = \frac{h_i(x,y)}{f_y(x,y)}$ as in~\eqref{eq:algbasis1}. This defines the local map in~\eqref{eq:happymap},
with $p=(x,y) \in C^o$.
Then $\mathcal{D}_C$ is obtained by acting with the group $G$ 
on the image curve in $\mathbb{WP}^{3g-1}$.
This action corresponds to changing local coordinates  on $C$. 
We now use this to find equations in $\mathcal{I}(\mathcal{D}_C)$.
For any homogeneous polynomial $F\in \mathbb{C}[U,V,W]$,
let $F_{|C}$ denote its pullback to $C$ via \eqref{eq:happymap}.

\begin{lemma}\label{lemma:sectioncanonical}
	Let $F \in \mathbb{C}[U,V,W]_d$ 	which 
	is a relative $G$-invariant,~i.e.
	\begin{equation}\label{eq:conditioninvariace}
	(\textsl{g}\cdot F)_{|C} \,\,=\,\, a^d (F_{|C}) \qquad \text{for all} \,\, \textsl{g} \in G. \vspace{-0.08cm}
	\end{equation} 
	\begin{enumerate} 
		\item If $F_{|C}=0$, then $F \in \mathcal{I}(\mathcal{D}_C)$. \vspace{-0.2cm}
		\item In general, there exists a polynomial $A(U) \in\mathbb{C}[U]_d$ such that $F-A(U) \in \mathcal{I}(\mathcal{D}_C)$.
	\end{enumerate}
\end{lemma}

\begin{proof}
The pullback $F_{|C}$ is an algebraic expression in terms of the local coordinate $x$. The relative 
$G$-invariance of \eqref{eq:conditioninvariace} means the following: if the local coordinate $x$ is changed to $z = \phi(x)$, then $F_{|C}$ is scaled by a power of $\phi'(0)$.
This derivative is exactly the cocycle corresponding to the canonical bundle $\omega_C$.
Hence  \eqref{eq:conditioninvariace} means that $F_{|C}$ 
represents a section in $H^0(C,\omega_C^d)$, independent of the
 local coordinate $x$. In particular, if $F_{|C}=0$, then
$F$ vanishes on the whole Dubrovin threefold~$\mathcal{D}_C$. This proves the first point in
Lemma \ref{lemma:sectioncanonical}.
	
To prove the second point, recall that $C$ is not hyperelliptic.
By Max Noether's Theorem, the multiplication map
$\operatorname{Sym}^d H^0(C,\omega_C) \longrightarrow H^0(C,\omega_C^d)$
is surjective. Since $u_1,\dots,u_g$ corresponds to the basis $\omega_1,\dots,\omega_g$ of $H^0(C,\omega_C)$,
	 there is a polynomial $A(U)\in \mathbb{C}[U]_d$ whose restriction to $C$ coincides with $F_{|C}$. Now if is enough to apply the first point to $F-A(U)$. 
\end{proof}

\begin{lemma} \label{lem:trailingterms} Fix a pair of indices $i,j$
	satisfying $1\leq i<j\leq g$.
	There exist homogeneous polynomials $A \in \CC[U]_3$ and $B \in \CC[U]_5$
	such that the following polynomials belong to $\mathcal{I}(\mathcal{D}_C)$:
	\begin{align}
		\begin{vmatrix} u_i\,\, & v_i \\ u_j\, & v_j \end{vmatrix} &\,-\, A(U), \label{eq:equationUV} \\
		\begin{vmatrix} u_i & w_i \\ u_j & w_j \end{vmatrix} &\,-\, \frac{1}{2}\sum_{h=1}^g \frac{\partial A}{\partial u_h}(U)\cdot v_h, \label{eq:equationUW} \\
		\begin{vmatrix} v_i\, & w_i \\ v_j & w_j \end{vmatrix} & \,+\,
		\frac{1}{3} \sum_{h=1}^g \frac{\partial A}{\partial u_h}(U)\cdot w_h 
		\,-\,\frac{1}{4} \sum_{h=1}^{g}\sum_{k=1}^g \frac{\partial^2\! A}{\partial u_h\partial u_k}(U)\cdot v_hv_k \,-\, B(U). 
		\label{eq:equationVW}
	\end{align}
\end{lemma}

\begin{proof}
	We start with the ansatz \eqref{eq:equationUV}. If $\textsl{g}\in G$ is as in \eqref{eq:groupG}, then 
	an easy computation shows
	\[
	\textsl{g} \cdot  \begin{vmatrix} u_i & v_i \\ u_j & v_j \end{vmatrix} \,\,=\,\, a^3 \begin{vmatrix} u_i & v_i \\ u_j & v_j \end{vmatrix} .
	\] 
	Hence, by Lemma \ref{lemma:sectioncanonical}, there exists 
	a polynomial $A(U)\in \mathbb{C}[U]_3$ such that 
	the difference \eqref{eq:equationUV} belongs to $\mathcal{I}(\mathcal{D}_C)$.
	We note that this can also be seen via the Gaussian maps of \cite{Wahl}.
	
	Next, for \eqref{eq:equationUW}, we use
	Euler's relation $\,A(U) = \frac{1}{3}\sum_{h=1}^g \frac{\partial A}{\partial u_i}(U)u_i$. 
	The action by $\textsl{g}$ gives
	$$	 \begin{small} \!
	\textsl{g}\cdot \left( \begin{vmatrix} u_i \! & \! w_i \\ u_j\! &  \! w_j \end{vmatrix}
	- \frac{1}{2}\sum_{h=1}^g \frac{\partial A}{\partial u_h}(U)v_h \!\right)
	= a^4\left( \begin{vmatrix} u_i \!&\! w_i \\ u_j \!&\! w_j \end{vmatrix}
	- \frac{1}{2}\sum_{h=1}^g \frac{\partial A}{\partial u_h}(U)v_h \! \right)
	+ \,3a^2b  \left( \begin{vmatrix} u_i \! &\! v_i \\ u_j \! & \! v_j \end{vmatrix}-A(U)
	\!\right).
	\end{small}	
	$$
	Restricting this identity to $C$ and using \eqref{eq:equationUV}, we see that the condition \eqref{eq:conditioninvariace} is satisfied. Furthermore, if we restrict \eqref{eq:equationUW} to $C$, then we get $1/2$ times 
	$$	
	\begin{vmatrix} H_i & \ddot H_i \\  H_j & \ddot H_j \end{vmatrix}
	\, +\, \sum_{h=1}^g \frac{\partial A}{\partial u_h}(H_1,\dots,H_g)\dot{H}_h
	\quad = \quad \frac{d}{dx} \left(\,
	\begin{vmatrix} H_i & \dot{H}_i \\ H_j & \dot H_j \end{vmatrix} \,+\, A(H_1,\dots,H_g) \,\right).
	$$
	The parenthesized expression on the right is the restriction of  \eqref{eq:equationUV} to $C$.
	It vanishes identically on $C$, and hence so does its derivative.
	Therefore, \eqref{eq:equationUW} lies in $\mathcal{I}( \mathcal{D}_C)$,
	by Lemma~\ref{lemma:sectioncanonical}.
	
	We conclude with \eqref{eq:equationVW}. We apply the group element $\textsl{g} \in G$ in \eqref{eq:actionG} to the polynomial
	\begin{equation}
		\label{eq:quinticsecfion}
		\begin{vmatrix} v_i\, & w_i \\ v_j & w_j \end{vmatrix}\,+\,\frac{1}{3} \sum_{h=1}^g \frac{\partial A}{\partial u_h}(U)\cdot w_h 
		\,-\,\frac{1}{4} \sum_{h=1}^{g}\sum_{k=1}^g \frac{\partial^2\! A}{\partial u_h\partial u_k}(U)\cdot v_h v_k.
	\end{equation}
	Using Euler's relation and its
	generalizations $\,\sum_{h=1}^g \sum_{k=1}^g \frac{\partial^2\! A}{\partial u_h \partial u_k} u_h v_k = 2\sum_{h=1}^g \frac{\partial A}{\partial u_h}(U)v_h\,$ and $\,\sum_{h=1}^g \sum_{k=1}^g \frac{\partial^2\! A}{\partial u_h \partial u_k}(U)u_hu_k = 6A(U)$,
	we find that the result of this application equals
%	\begin{align*}
$$ \!\!\! \!\begin{matrix}	
		&  a^5 \biggl( \, \begin{vmatrix} v_i\, & w_i \\ v_j & w_j \end{vmatrix}+\frac{1}{3} \sum_{h=1}^g \frac{\partial A}{\partial u_h}(U)\cdot w_h -\frac{1}{4} \sum_{h=1}^{g}\sum_{k=1}^g \frac{\partial^2\! A}{\partial u_h \partial u_k}(U)\cdot v_h v_k \biggr) \qquad \qquad \qquad \quad \\
		& \! +\, 6ab^2 \left( \,\begin{vmatrix} u_i \! &\! v_i \\ u_j \!& \! v_j \end{vmatrix} - A(U) \! \right)
		\,-\, a^2 c \,\,\,\left( \,\begin{vmatrix} u_i \! &\! v_i \\ u_j \! & \! v_j \end{vmatrix} - A(U)\! \right)
		\, +\,2a^3 b \left( \,\begin{vmatrix} u_i \!&\! w_i \\ u_j\! & \! w_j \end{vmatrix} 
		\,-\, \frac{1}{2} \sum_{h=1}^g \frac{\partial A}{\partial u_h}(U) v_h \!\right).  
%	\end{align*}
\end{matrix} $$ 
		The last three parentheses agree with \eqref{eq:equationUV},\,\eqref{eq:equationUW}	
	and are hence zero on $C$. This implies that~\eqref{eq:quinticsecfion} is a relative
	$G$-invariant on $C$.
	Applying Lemma \ref{lemma:sectioncanonical} again concludes the proof.
\end{proof}

\begin{lemma}\label{lemma:polarization}
The polarization  of the canonical ideal belongs to the initial ideal of $\,\mathcal{I}(\mathcal{D}_C)$.
\end{lemma} 

\begin{proof}
	Let $f\in \mathcal{I}(C)_d$ be a homogeneous polynomial of degree $d$ in the canonical ideal.
	We view $f$ as a symmetric tensor of order $d$ in $g$ variables. Then $f(xU+yV+tW) = \sum_{a+b+c=d}f(U^a\otimes V^b\otimes W^c)\cdot x^ay^bt^c$. We need to prove that 
	all coefficients $f(U^a\otimes V^b\otimes W^{c})$ belong to the initial ideal of $\mathcal{I}(\mathcal{D}_C)$. If is enough to do so when $f$ is a generator of $\mathcal{I}(C)$.
	By Petri's theorem, we must consider quadrics and cubics. Quartics are covered by 
	Theorem~\ref{thm:quartic}. 

Suppose $d=2$.	
We claim that there are 
homogeneous polynomials $A_{110},A_{020},A_{101}, A_{011}$, $ A_{002}$ 
in $ \mathbb{C}[U]$ of degrees $3,4,4,5,6$ respectively such that the following six belong to $\mathcal{I}(\mathcal{D}_C)$:
$$	\begin{small} \begin{matrix}
	f(U\otimes U)\, ,\,\,
	f(U\otimes V) - A_{110}(U^3) \, , \\
	f(V\otimes V) - 2A_{110}(U^{2}\otimes V) - A_{020}(U^4) \, , \,\,
	f(U\otimes W) -\frac{3}{2}A_{110}(U^2\otimes V)-A_{101}(U^4) \, , \\
	f(V\otimes W) -A_{110}(U^2\otimes W) - \frac{3}{2}A_{110}(U\otimes V^2)-\frac{3}{2}A_{020}(U^3\otimes V)-A_{001}(U^3\otimes V)-A_{011}(U^5)\, ,\\
	f(W\otimes W) -3A_{110}(U {\otimes} V {\otimes} W) - 3A_{011}(U^4 {\otimes} V)
	-2A_{101}(U^3 {\otimes} W)-\frac{9}{4}A_{020}(U^2 {\otimes} V^2)-A_{002}(U^6).
	\end{matrix} \end{small}
$$
	The first polynomial
	 $f(U\otimes U)$ belongs to $\mathcal{I}(\mathcal{D}_C)$ by definition. Acting
	 with $\textsl{g} \in G$ on the second polynomial,
	  we obtain $\textsl{g}\cdot f(U\otimes V) = 2abf(U\otimes U)+a^3 f(U\otimes V)$.
	  The restriction to $C$ satisfies the condition \eqref{eq:conditioninvariace}.
Hence $f(U\otimes V)-A_{110}(U^3) \in \mathcal{I}(\mathcal{D}_C)$
 for some $A_{110}\in \mathcal{C}[U]_3$.
   
We next consider the third polynomial $f(V\otimes V)$.
A computation reveals
$$	\begin{matrix} 
	& \textsl{g}\cdot \left( f(V\otimes V) - 2A_{110}(U^{2}\otimes V) \right)
 \qquad \qquad   \qquad \qquad \smallskip \\ = & 4b^2 \cdot f (U {\otimes} U) 
	 +\,4a^2b \cdot \left( f(U {\otimes} V) - A_{110}(U^3)\right)
	+a^4\cdot \left(f (V {\otimes} V)-2A_{110}(U^2 {\otimes} V) \right)  .
\end{matrix} $$
Restricting this to the curve $C$ and using what is already proven,
we see that \eqref{eq:conditioninvariace} is satisfied.
By Lemma \ref{lemma:sectioncanonical}, we have  
$f(V {\otimes} V) - 2A_{110}(U^{2} {\otimes} V) - A_{020}(U^4) \in \mathcal{I}(\mathcal{D}_C)$
for some $A_{020} \in \mathbb{C}[U]_4$. The 
remaining three equations can be verified in an analogous way. 

The same reasoning works for  $d=3$. Let $f \in  \mathcal{I}(C)_3$
be a cubic that vanishes on the canonical curve.
Then there exist polynomials $A_{ijk}$ 
such that the following ten expressions are in $\mathcal{I}(\mathcal{D}_C)$.
The derivation of the  $A_{ijk}$ is analogous to the $d=2$ case
and to Lemma  \ref{lem:trailingterms}.
$$ \begin{small} \begin{matrix}
f(U^3) \, , \,\,\,
f(U^2\otimes V)  -A_{210}(U^4) \, , \,\,\,
f(U\otimes V^2)  - 2A_{210}(U^3\otimes V) - A_{120}(U^5) \, , \\
f(V^3)   -3A_{210}(U^2 {\otimes} V^2)-3A_{120}(U^4 {\otimes} V)-A_{030}(U^6) \,,\,\,
f(U^2 {\otimes} W)   -\frac{3}{2}A_{210}(U^3 {\otimes} V)-A_{201}(U^5) \, ,   \\
\,\,\, f(U {\otimes} V {\otimes} W)  -A_{210}(U^3 {\otimes} W)-\frac{3}{2}A_{210}(U^2{\otimes} V^2) - A_{201}(U^4{\otimes} V) - \frac{3}{2}A_{120}(U^4{\otimes} V)-A_{111}(U^6) \, , \\
f(V^2{\otimes} W)   -\frac{3}{2}A_{210}(U{\otimes} V^3)-2A_{210}(U^2{\otimes} V{\otimes} W)-A_{201}(U^3{\otimes} V^2)   - 3A_{120}(U^3{\otimes} V^2) \qquad \quad
\\ -\,A_{120}(U^4{\otimes} W)-2A_{111}(U^5{\otimes} V)-\frac{3}{2}A_{030}(U^3{\otimes} V) 
- A_{021}(U^7) \, , \\
f(U{\otimes} W^2)   -3A_{210}(U^2{\otimes} V{\otimes} W)-\frac{9}{4}A_{120}(U^3{\otimes} V^2)-2A_{201}(U^4{\otimes} W)-3A_{111}(U^5{\otimes} V)-A_{102}(U^7) \, , \quad \\
\end{matrix} \end{small} $$
$$ \begin{small} \begin{matrix}
\! f(V{\otimes} W^2)   {-} 3A_{201}(U{\otimes} V^2{\otimes} W) {-} A_{210}(U^2{\otimes} W^2)
{-} 2A_{201}(U^3{\otimes} V{\otimes} W)  {-} 2A_{111}(U^5{\otimes} W) 
 {-} \frac{9}{4}A_{120}(U^3{\otimes} V^3)\\ \quad -3A_{120}(U^3{\otimes} V{\otimes} W) {-}3A_{111}(U^4{\otimes} V^2)  {-} \frac{9}{4}A_{030}(U^4{\otimes} V^2){-} 3A_{021}(U^6{\otimes} V)
 {-}A_{102}(U^6{\otimes} V){-}A_{012}(U^8)\, , \\
f(W^3)   -\frac{9}{2}A_{210}(U{\otimes} V{\otimes} W^2)-3A_{201}(U^3{\otimes} W^2)-\frac{27}{4}A_{120}(U^2{\otimes} V^2{\otimes} W)   - 9A_{111}(U^3{\otimes} V^2{\otimes} W) \qquad \quad
\\ \quad \,\, - \frac{27}{8}A_{030}(U^3{\otimes} V^3)-\frac{27}{4}A_{021}(U^5{\otimes} V^2) -3A_{102}(U^5{\otimes} W)-\frac{9}{2}A_{012}(U^7{\otimes} V)-A_{003}(U^9).
\end{matrix} \end{small}
$$
This completes the proof of Lemma~\ref{lemma:polarization}.
\end{proof}

\begin{proof}[Proof of Theorem \ref{thm:initialideal}]
The canonical ring $\CC[U]/I_C$ is an integral domain. We
consider its Segre product 
with a polynomial ring in three variables. This is the quotient
of the polynomial ring $\CC[U,V,W]$ in $3g$ unknowns modulo
a prime ideal $K$. The ideal $K$ is described in several sources, including the 
second textbook by Kreuzer and Robbiano  \cite[Tutorial~82]{KR}.
We also refer to Sullivant \cite[\S 3.1]{Seth} who offers a more
general construction of toric fiber products, along with a recipe for
lifting generators and Gr\"obner bases from $I_C$ to~$K$.

The Segre product ideal $K$ is precisely our ideal on the right hand side in (\ref{eq:inIDC}).
The irreducible affine variety in $\CC^{3g}$ defined by $K$ has dimension $4$.
Indeed, the point $U$ lies in the cone over
the curve $C$, and $V$ and $W$ are multiples of $U$, so there are four
degrees of freedom in total.
The only difference to the standard setting in \cite{KR, Seth} is our grading, 
with degrees $1,2,3$ for $U, V,W$ respectively.
We conclude that $K$ defines a threefold in $\mathbb{WP}^{3g-1}$.

We next claim that the initial ideal
${\rm in}\bigl(\mathcal{I}(\mathcal{D}_C) \bigr)$ contains $K$.
For every generator of $K$, we must find a polynomial with all terms of lower 
weight which is congruent modulo  $\mathcal{I}(\mathcal{D}_C)$ to that generator.
For the $2 \times 2 $ minors of the $g \times 3$-matrix $(UVW)$,
this is precisely the content of Lemma \ref{lem:trailingterms}.
For example, in (\ref{eq:jacobrel}) the generator $u_k v_l - u_l v_k$
is congruent to  $\pm J_{kl}$.
In general, these are the trailing terms in
\eqref{eq:equationUV},
\eqref{eq:equationUW} and \eqref{eq:equationVW}.
For the polarizations of the generators of the canonical ideal $I_C$,
the trailing terms are constructed in Lemma \ref{lemma:polarization}.

Now, we know that  $\mathcal{I}(\mathcal{D}_C) $ is prime of
dimension $4$, and hence ${\rm in}\bigl(\mathcal{I}(\mathcal{D}_C) \bigr)$
has dimension~$4$. But, it need not be radical and it could have embedded
components. However, it contains the prime ideal $K$
of the same dimension. This implies that 
$\,K = {\rm in}\bigl(\mathcal{I}(\mathcal{D}_C) \bigr)\,$ as desired.
\end{proof}

\section{Degenerations} \label{sec6}

A standard technique for studying smooth algebraic curves 
is to replace them with curves that are  singular and reducible.
Such degenerations are central to the theory of moduli spaces.
The study of moduli of curves is a vast subject, with lots of beautiful combinatorics.
Keeping this broader context in the back of our minds, we here ask the following question:
$$ \text{\em What happens to the  
Dubrovin threefold $\,\mathcal{D}_C$ when the curve $C$ degenerates?} $$

Our aim in this section is to take first steps towards answering that question.
We focus on two classes of degenerations. We describe these by
their effects on the Riemann theta function $\theta({\bf z})$ associated with the curve.
First, there are the degenerations that are visible in the
Deligne-Mumford moduli space $ \overline{\mathcal{M}}_g$.
We refer to them as {\em tropical degenerations}.
 These turn the infinite sum
on the right hand side of \eqref{eq:RTFreal} into a finite sum of exponentials.

Second, there is the class of {\em node-free degenerations}, which replace
 the Riemann theta function by polynomials.
These polynomial theta functions were characterized for genus three by Eiesland
 \cite{Eies09}. His list was studied computationally in our recent work~\cite[\S 5]{struwe}.
These degenerations lead to rational solutions of Hirota's equation \eqref{eq:hirota}, 
and these give soliton solutions of the KP equation \eqref{eq:KP1}.
We expect interesting connections to the theory in~\cite{Kodamabook}.

We begin with the first class, namely the tropical degenerations.
A  {\em rational nodal curve} $C$ is a stable curve of genus $g$ whose
irreducible components are rational. Stability implies that all singularities are nodes.
The dual graph $\mathcal{G}$ of $C$ has one vertex for each irreducible component
and one edge for each node. This edge is a loop when this node is a singular
point on one irreducible component. Two vertices of $\mathcal{G}$ can be
connected by multiple edges, namely when two irreducible components
of $C$ intersect in two or more points. The hypothesis that each irreducible
component is rational implies that  $\mathcal{G}$ is trivalent and it has
$2g-2$ vertices and $3g-3$ edges. Up to isomorphism, the number of such trivalent graphs is
$2,5, 17, 71, 388, \ldots$ when $g=2,3,4,5,6,\ldots$. The tropical Torelli map \cite[\S 6]{BBC}
contracts all bridges in the graph $\mathcal{G}$. Combinatorially, it is the map
that takes $\mathcal{G}$ to its  corresponding cographic matroid.

From the cographic matroid, one derives the Voronoi subdivision of $\RR^g$.
It is dual to the Delaunay subdivision \cite[\S 5]{BBC}, which is the regular polyhedral subdivision 
of $\ZZ^g$ induced by a quadratic form given by the Laplacian of $\mathcal{G}$.
The Voronoi cell is the set of all points in $\RR^g$ whose closest lattice point
is the origin. This $g$-dimensional polytope belongs to the
class of unimodular zonotopes.  The possible combinatorial types of Voronoi cells are listed in
\cite[Figure 4]{struwe} for $g=3$ and in \cite[Table 1]{CKS} for $g=4$.
Every vertex ${\bf a}$ of the Voronoi cell is dual to a Delaunay polytope.
We write $V_{\bf a} \subset \ZZ^g$ for the set of vertices of this Delaunay polytope.

We write $V_{\bf a} \subset \ZZ^g$ for the set of vertices of this Delaunay polytope.
With this we associate the following function, given
by a finite exponential sum with  certain coefficients $\gamma_{\bf u} \in \mathbb{C}$:
\begin{equation}\label{eq:exponentialsum}
\theta_{C,\mathbf{a}}({\bf z}) \quad =\quad 
 \sum_{{\bf u} \in V_{\bf a}} \gamma_{\bf u} \cdot {\rm exp}( {\bf u}^T {\bf z} ) .
\end{equation}
The following was shown in \cite[Theorem 4.1]{struwe} for genus $g=3$,
that is, for quartics $C$ in $\PP^2$.

\begin{proposition} \label{thm:exponentialsum}
  Consider a rational nodal quartic $C$ and a vertex ${\bf a}$ of the Voronoi cell as above. There
  exists a choice of coefficients $\gamma_{\bf u} \in \mathbb{C}$ such that  the function 
  $\theta_{C,\mathbf{a}}$ is a limit of translated Riemann theta functions 
  associated with a family of smooth quartics with limit  $C$.
  \end{proposition}

The proof in \cite{struwe} uses a linear family of Riemann matrices
and it does not extend to higher genus. However, the statement should be
true for all $g \geq 4$, and it should follow from known results
about degenerations of Jacobians. See also the discussion in
\cite[Lemma 4.2]{Little}.
 Another approach to a proof
 is the use of non-Archimedean geometry as in~\cite{FRSS}.
The tropical limit process in \cite[\S 4.3]{FRSS} can be viewed as a
flat family with special fiber~$C$.
 Moreover, we believe that a converse statement holds, namely that every
flat family of smooth curves which degenerates to a rational nodal curve $C$
induces a truncated theta series of the form~(\ref{eq:exponentialsum}).

There are two extreme cases of special interest. If $C$ is a rational curve
with $g$ nodes then the Delaunay polytope is a cube and the
 theta function $\theta_C$ is an exponential sum of $2^g$ terms
(cf.~\cite[Example 4.3]{struwe}).
This case corresponds to soliton solutions of the KP equation~\cite{Kodamabook}.

At the other end of the spectrum are {\em graph curves} \cite{BayEis}.
These are curves whose canonical model consists of $2g-2$ straight lines in $\PP^{g-1}$.
The canonical ideal $I_C$ of a graph curve  $C$ has a combinatorial description, given in \cite[\S 3]{BayEis}.
The corresponding trivalent graph $\mathcal{G}$ is simple, and it has no loops or multiple edges.
For instance, for $g=3$ this implies that $\mathcal{G}$ equals $K_4$, the complete graph on
four nodes. The associated theta function is a sum of four terms, one for each vertex of
the Delaunay tetrahedron. Namely, in \cite[equations (29) and (54)]{struwe} we find
\begin{equation}
\label{eq:tetrahedraltheta}
 \theta_C({\bf z}) \,\,= \,\, \gamma_0 \,+\, \gamma_1 \cdot {\rm exp}(z_1) \,+\,
 \gamma_2 \cdot {\rm exp}(z_2) \,+\, \gamma_3 \cdot {\rm exp}(z_3). 
 \end{equation}
For $g=4$ there are two types of graph curves, namely $\mathcal{G}$ is either the bipartite graph $K_{3,3}$
or the edge graph of a triangular prism. Their theta functions are truncations as in
(\ref{eq:exponentialsum}).

\begin{example}[Four lines in $\PP^2$] \label{ex:fourlinesinP2}
Let $g=3 $ and consider the plane quartic $C$ defined by
\begin{equation}
\label{eq:fuuuu}
 f \quad = \quad u_2 u_3 (u_2 - u_1) (u_3 - u_1).
\end{equation}
This is a graph curve with $\mathcal{G} = K_4$. One approach to defining a
Dubrovin threefold $\mathcal{D}_C$ in $\mathbb{WP}^8$ is to use Theorem \ref{thm:quartic}.
The ideal $I$ given there is the
intersection of four prime ideals:
\begin{equation}
\label{eq:fourprimeideals}
 \begin{matrix}
I & = & \quad
(I + \langle u_2,v_2,w_2\rangle) \,\cap \,
(I + \langle u_2-u_1,v_2-v_1,w_2-w_1\rangle) \\  & & \cap \, \,\,
(I + \langle u_3,v_3,w_3\rangle) \,\cap \,
(I + \langle u_3-u_1,v_3-v_1,w_3-w_1\rangle) .
\end{matrix}
\end{equation}
Each associated prime has six minimal generators, of degrees $1,2,3,3,4,5$. For instance,
$$ I + \langle u_2,v_2,w_2\rangle \,\, = \,\,
\langle\, u_1\, ,\, v_1\,,\, w_1\,,\,
 u_3 v_1 + u_1 v_3 + \cdots,\, 
 u_3 w_1 + u_1 w_3 + \cdots,\, 
 v_3 w_1 + v_1 w_3 + \cdots \,\rangle. $$
 Remarkably, the radical ideal $I$ has $17$ minimal 
generators, of precisely the degrees promised in Theorem \ref{thm:quartic}.
Here, tropical degeneration to $f$ gives  a flat family of Dubrovin threefolds.

A second approach is to apply the PDE method in Section \ref{sec4}.
For each ${\bf z} \in \CC^3$ we can define a Hirota quartic $H_{\bf z}$
use the tetrahedral theta function  in \eqref{eq:tetrahedraltheta}.
We find that $H_{\bf z}$ equals
\begin{equation}
\label{eq:fourlineshirota}
\begin{matrix}
  &8d\gamma_0^2\,+\,8d\gamma_1^2 \cdot {\rm exp}(2z_1)\,+\,
  8d\gamma_2^2 \cdot {\rm exp}(2z_2)\,+\,8d\gamma_3^2 \cdot {\rm exp}(2z_3)\\
&+\sum_{i=1}^3\gamma_0\gamma_i(u_i^4+6cu_i^2+3v_i^2-4u_iw_i+16d) \cdot{\rm exp}(z_i)\\
&+\sum_{1\leq i<j \leq 3}\gamma_i\gamma_j(u_i^4-4u_i^3u_j+6u_i^2u_j^2+6cu_i^2-4u_iu_j^3-12cu_iu_j+u_j^4  -4u_iw_i 
\\ &\quad \quad \quad\quad
+4u_iw_j
+6cu_j^2+4w_iu_j-4w_ju_j+3v_i^2-6v_iv_j+3v_j^2+16d) \cdot {\rm exp}(z_i+z_j).
\end{matrix}
\end{equation}
For the specific quartic $f$ in \eqref{eq:fuuuu}, we have $\gamma_0=\gamma_1= \gamma_2=1, \gamma_3=-1$,
 by \cite[Example 3.3]{struwe}.
 The ideal $\langle H_{\bf z} : {\bf z} \in \CC^3 \rangle $ is generated by
 $d$ and the six other coefficients of the exponential terms.
It defines a variety in $\mathbb{WP}^{10}$, but this now
has dimension four and is reducible. Its projection into the
$\mathbb{WP}^{8}$ with coordinates $U,V,W$ decomposes into two isomorphic schemes,
corresponding to the union  $\,\mathcal{D}_C \cup \mathcal{D}_C^{-}\, $ in Remark \ref{rem:theothercomponent}.
Each of these two pieces is reducible. Indeed, we found many components of dimension three and four.
These deserve further~study.

Finally, our third method is to use the algebraic parametrization in  Section~\ref{sec3}. 
The basis \eqref{eq:differentialBasisTrott} for $H^0(C, \Omega^1_C)$
is valid also for reducible curves, possibly after a linear change of coordinates. We consider the image of
 $C \rightarrow \mathbb{WP}^8$  in~\eqref{eq:happymap},
and we compute its orbit under the group $G$.
After a linear change of coordinates, then
the resulting threefold has four irreducible components,
and its radical ideal coincides with \eqref{eq:fourprimeideals}.
 \end{example}

\begin{example}[Six lines in $\PP^3$] \label{ex:sixlinesinP3}
Let $g=4 $ and consider the space sextic $C$ defined by
$$ q \,\, = \,\, u_1 u_4 - u_2 u_3 \quad {\rm and} \quad
 f  \,=\,  u_1 u_2 u_3-u_2^2 u_4-u_3^2 u_4+u_4^3 . $$
 In spite of $q$ and $f$ being irreducible,
their ideal $\langle q,f \rangle$ decomposes.
This is a graph curve, with $\mathcal{G} = K_{3,3}$.
We consider the three approaches  to $\mathcal{D}_C $
as in Example \ref{ex:fourlinesinP2}. 
The description in Example~\ref{ex:gennuss4} leads
to a radical ideal that is  the intersection of six prime ideals, 
analogously to~\eqref{eq:fourprimeideals}.
Taking $f = (x^3-x)(y^3-y)$ in \eqref{eq:affine33},
we obtain a rational map into $  \mathbb{WP}^{11}$.
The orbit of the image under $G$ is a threefold with only three irreducible components.
However, after a general linear change of coordinates, we recover
all six components, and the two ideals~agree.
\end{example}

\smallskip

We now turn to node-free degenerations. 
By a {\em rational node-free curve} we mean a reduced curve $ C$  of geometric genus $g$ whose
components are rational and none of whose singularities are nodes.
One example is an irreducible rational curve whose singularities
are cusps.  Yet, typical examples are reducible, such as the
cuspidal cubic together with its cuspidal tangent in \cite[Example 3.1]{struwe}.
Node-free degenerations were
classified for genus three by Eiesland \cite{Eies09}.
Based on his work, and our recent follow-up in \cite[\S 5]{struwe}, we 
believe that the following holds.

\begin{conjecture} \label{thm:thetapolynomial}
Consider a flat family of smooth curves which degenerates to a rational node-free curve $C$.
In its limit,  the Riemann theta function converges to a polynomial in ${\bf z} $.
\end{conjecture}

In what follows we focus on the case of genus three.
The classification of node-free degenerations by Eiesland \cite{Eies09}
was carried out in the context of  double translation surfaces.
This subject was initiated by Lie in the 19th century.  In our recent study \cite{struwe},
we use the term {\em theta surfaces} for what is essentially the
zero set of  the theta function in $\CC^3$. We refer to the work of
Little \cite{Little} for a 20th century generalization of Lie's theory to higher genus.

Consider a reduced, but possibly singular, plane
quartic curve $C\subset \mathbb{P}^2$, given by an affine equation $f(x,y)=0 $.  Around each smooth point of $C$, with $f_y\ne 0$, we can consider the differentials
$\omega_1,\omega_2,\omega_3$ in \eqref{eq:differentialBasisTrott}.
Fix two distinct smooth points $p_0,q_0 \in C$. For points $p$ and $q$ moving in small neighborhoods of $p_0$ and $q_0$ respectively in the curve $C$, we consider the~map
\begin{equation}
\label{eq:pqmap}
 (p,q)\,\, \mapsto \,\,\left( \int_{p_0}^p \! \omega_1 \,+\, \int_{q_0}^q \!\omega_1\, ,\,\,
 \int_{p_0}^p \!\omega_2 \,+\, \int_{q_0}^q \!\omega_2 \, , \,\,
  \int_{p_0}^p \!\omega_3 \,+\, \int_{q_0}^q \!\omega_3 \right) .
\end{equation}  
  Its image in $\mathbb{C}^3$ is the theta surface of $C$.
  It satsifies an analytic equation $\theta_C(z_1,z_2,z_3)=0$.  For smooth curves $C$,
    the function $\theta_C$ equals the classical Riemann theta function, 
    up to an affine change of coordinates that is similar to \eqref{eq:basistransf}.
    If the curve $C$ is singular, then $\theta_C$ is a degenerate theta function, as in
equation  (\ref{eq:exponentialsum})  or Conjecture
\ref{thm:thetapolynomial}.     We would like to
  define the Dubrovin threefolds
  $\mathcal{D}_C$ and $\mathcal{D}^{\rm big}_C$ as in Sections 
  \ref{sec:Intro}--\ref{sec4}, either by a parametrization
  or by implicit equations. But this is subtle, as shown for graph curves in Examples 
  \ref{ex:fourlinesinP2} and~\ref{ex:sixlinesinP3}.

We studied this issue experimentally for Eiesland's curves in
 \cite{Eies09}. Their
theta functions are polynomials  of degrees $3,4,5,6$.
See \cite[\S 5]{struwe} for pictures of these polynomial theta surfaces in~$\RR^3$.
Here is a concrete example of a node-free curve and its polynomial theta function.

\begin{example} Following \cite[Example 5.3]{struwe}, we consider the quartic curve $C$ in $\PP^2$
with affine equation $f=x^4-y^3$. The unique singular point  $(0,0)$ is not a node.
Using the  rational parametrization $x=t^{-3},\,y=t^{-4}$, 
we write our three differentials with local coordinate $t$ as 
$\,\omega_1=t^4 \cdot dt, \,\omega_2=t \cdot dt,\, \omega_3=1 \cdot dt$.
This implies that the theta surface has the parametrization
$$  z_1 \,=\, \frac{1}{5}(p^5 + q^5)\,, \,\,\, z_2 \,=\, \frac{1}{2}(p^2+q^2)\, ,\,\,\, z_3 \,=\, p+q. $$
The parameters $p$ and $q$ correspond to the
integration limits in (\ref{eq:pqmap}), via a 
slight abuse of notation.
See also \cite[Example 4.3]{Little}.
We conclude that the theta function of $C$ is the quintic
$$ \theta_C(z_1,z_2,z_3) \,\,=\,\, z_3^5\,-\,20z_2^2z_3\,+\,20z_1. $$

We examine our three approaches to the Dubrovin threefold
as in Example~\ref{ex:fourlinesinP2}.
We begin with Theorem \ref{thm:quartic} for the homogeneous quartic
$f = u_1^4 - u_2^3 u_3$.
The ideal given there is a flat degeneration of the general case. It
is a prime ideal, with $17$ minimal generators as before, 
starting with
$\,
u_2^3+u_2 v_1-u_1 v_2,
3 u_2^2 u_3-u_3 v_1+u_1 v_3,
4 u_1^3-u_3 v_2+u_2 v_3,
12 u_1^2   v_1-2 u_3 w_2+2 u_2 w_3 ,\ldots $

Our second approach is to apply the PDE method of Section \ref{sec4}, but in arithmetic over~$\QQ$.
The ideal generated by the Hirota quartics $H_{\bf z}$ defines a threefold in $\mathbb{WP}^{10}$. 
By eliminating $c$ and $d$, we obtain a threefold in $\mathbb{WP}^8$ which is also a candidate for the
Dubrovin threefold. It has two components which map to each other under the involution $(U,V,W)\mapsto (U,-V,W)$.
One of the two components agrees set-theoretically with that given by the prime ideal above.

Our third method is to use an algebraic parametrization 
based on the differential forms $\omega_1,\omega_2,\omega_3$ as in Section \ref{sec3}.
This leads to the same prime ideal with $17$ generators in $\QQ[U,V,W]$.
\end{example}   

\begin{remark}
The KP solutions with polynomial theta functions should be compared with the {\em lump solutions}
in \cite{AblLump}. There is surely a lot more to be said about the degenerations~above.
\end{remark}

\bigskip

\noindent {\bf Acknowledgements}.
We thank Bernard Deconinck for many inspiring conversations during a 
visit by D.A. to Seattle. D.A. thanks the Department of Applied Mathematics of the UW for its hospitality and MATH+ Berlin for 
financial support. We are grateful to Atsushi Nakayashiki and Farbod Shokrieh for helpful discussions
on topics in this article.
Many thanks also to Aldo Conca and Seth Sullivant for supplying references on 
Segre products.

\bigskip \medskip

\noindent
\footnotesize 
{\bf Authors' addresses:}

\smallskip

\noindent Daniele Agostini,  Humboldt-Universit\"at zu Berlin,
\hfill  {\tt daniele.agostini@math.hu-berlin.de}

\noindent 
 T\"urk\"u \"Ozl\"um \c{C}elik,  Universit\"at Leipzig and MPI-MiS Leipzig,
\hfill {\tt turkuozlum@gmail.com}

\noindent Bernd Sturmfels,
MPI-MiS Leipzig and UC Berkeley,
\hfill {\tt bernd@mis.mpg.de}

\end{document}